\documentclass[journal]{IEEEtran}

\usepackage{graphics} 
\usepackage{epsfig} 
\usepackage{amsmath} 
\usepackage{amssymb}  
\usepackage{amsthm}
\usepackage{multirow}
\usepackage{color}
\usepackage{cite}
\usepackage{url}
\usepackage{algorithm}
\usepackage[noend]{algpseudocode}
\usepackage[caption=false,font=footnotesize]{subfig}
\usepackage{graphicx}
\usepackage{epstopdf}
\usepackage{enumerate}

\newtheorem{theorem}{Theorem}

\theoremstyle{definition}
\newtheorem{definition}{Definition}
\newtheorem{assumption}{Assumption}
\newtheorem{problem}{Problem}
\newtheorem{example}{Example}

\newcommand{\upperbound}{\textsc{ApproxVerification}}
\newcommand{\supervisor}{\textsc{Supervisor}}

\begin{document}
\title{Safety Verification and Control for Collision Avoidance at Road Intersections}
%
%
%

\author{
	Heejin Ahn and Domitilla Del Vecchio
	\thanks{This work was in part supported by NSF Award \#1239182.}
	\thanks{Heejin Ahn and Domitilla Del Vecchio are with the Department of Mechanical Engineering, Massachusetts Institute of Technology, 77 Massachusetts Avenue, Cambridge, USA. 
		Email: {\tt\small hjahn@mit.edu} and {\tt\small ddv@mit.edu}}%
}

\maketitle

\begin{abstract}
This paper presents the design of a supervisory algorithm that monitors safety at road intersections and overrides drivers with a safe input when necessary. The design of the supervisor consists of two parts: safety verification and control design. Safety verification is the problem to determine if vehicles will be able to cross the intersection without colliding with current drivers' inputs. We translate this safety verification problem into a jobshop scheduling problem, which minimizes the maximum lateness and evaluates if the optimal cost is zero. The zero optimal cost corresponds to the case in which all vehicles can cross each conflict area without collisions. Computing the optimal cost requires solving a Mixed Integer Nonlinear Programming (MINLP) problem due to the nonlinear second-order dynamics of the vehicles. We therefore estimate this optimal cost by formulating two related Mixed Integer Linear Programming (MILP) problems that assume simpler vehicle dynamics. We prove that these two MILP problems yield lower and upper bounds of the optimal cost. We also quantify the worst case approximation errors of these MILP problems. We design the supervisor to override the vehicles with a safe control input if the MILP problem that computes the upper bound yields a positive optimal cost. We theoretically demonstrate that the supervisor keeps the intersection safe and is non-blocking. Computer simulations further validate that the algorithms can run in real time for problems of realistic size.
\end{abstract}


%
\IEEEpeerreviewmaketitle

\section{Introduction}
%
%
%
%
\IEEEPARstart{T}{he} first fatality caused by a self-driving technology has raised concerns about the safety of autonomous vehicles \cite{ackerman_fatal_2016}. As an approach to ensuring safety particularly at a road intersection, this paper proposes the design of a safeguard, called a supervisory algorithm or supervisor. The supervisor monitors vehicles' current states and inputs through vehicle-to-vehicle and vehicle-to-infrastructure communications \cite{US:2015:ITS}, and determines if their inputs will cause collisions. If this is the case, the supervisor intervenes to prevent collisions. 


Informally, we state safety verification as a problem that determines if the state trajectory can be kept outside an unsafe set given an initial condition, where the unsafe set is defined as the set of states corresponding to a collision configuration. Reachability analysis has been used to solve this problem by calculating the reachable region of a system to find a set of initial states that can be controlled to avoid the unsafe set \cite{tomlin_conflict_1998,tomlin_game_2000,lygeros_controllers_1999}. However, reachability analysis of dynamical systems with large state spaces is usually challenging due to the complexity of computing reachable sets. This motivates the development of several approximation approaches.
One approximation approach is to consider a simpler dynamical model to compute reachable sets instead of using the original complex dynamical model, as studied in \cite{henzinger_algorithmic_1998,alur_predicate_2006,girard_approximate_2007}. Another approximation approach is to approximate the original reachable set by employing various geometric representations, which include polyhedra \cite{asarin_approximate_2000}, ellipsoids \cite{ellipsoidal_2014}, or parallelotopes \cite{dreossi_parallelotope_2016}. It has been shown in \cite{moor_abstraction_2002} that monotonicity of system dynamics, for which state trajectories preserve a partial ordering on states and inputs, makes the reachability analysis relatively simple. This is because for such systems, the boundary of a reachable set can be computed by considering only maximum and minimum states and inputs \cite{del_vecchio_separation_2009}. Indeed, an exact method for the reachability analysis is presented in \cite{hafner_computational_2011} for piecewise continuous and monotone systems.

\begin{figure}[t!]
\centering
\includegraphics[width=\linewidth]{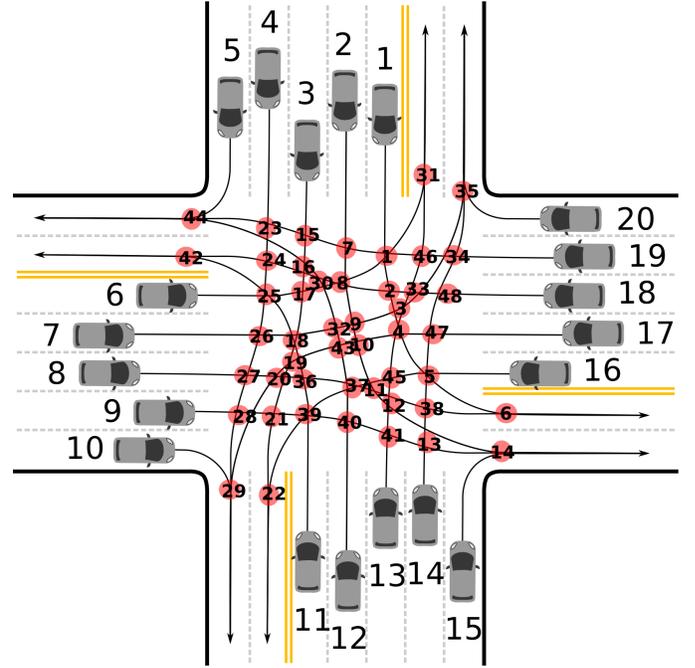}
\caption{General intersection. The safety verification problem in this scenario can be approximately solved within quantified bounds. This intersection is obtained from \cite{MassDOT_2012_Topcrash} to encompass 20 top crash locations in Massachusetts, USA.}
\label{figure:general_intersection}
\end{figure}

Our approach relies on the monotonicity of the system and on the approximation of vehicle dynamics. In this paper, we consider complex intersection scenarios in which vehicles follow predefined paths as shown in Figure~\ref{figure:general_intersection}. The longitudinal dynamics of vehicles are piecewise continuous and monotone \cite{hafner_computational_2011}, which enables us to translate the safety verification problem to a scheduling problem. This scheduling problem minimizes the maximum lateness and determines if the optimal cost is zero. The zero optimal cost implies that every job (vehicle) can be processed on each machine (conflict area) in time, and equivalently, vehicles can cross the intersection without collision. However, because of the nonlinear second-order dynamics of vehicles, the scheduling problem is a Mixed Integer Nonlinear Programming (MINLP) problem, which is computationally difficult to solve. We therefore estimate the optimal cost of the scheduling problem by formulating two Mixed Integer Linear Programming (MILP) problems that assume first-order dynamics and nonlinear second-order dynamics on a restricted input set, respectively. We prove that these MILP problems yield lower and upper bounds of the optimal cost. These lower and upper bound problems are equivalent to computing over- and under-approximation of the reachable sets of the original problem. These approximation bounds are quantified in this paper.

The scheduling problem has been employed to solve the safety verification problem for collision avoidance at an intersection \cite{colombo_efficient_2012,colombo_least_2014,bruni_robust_2013,ahn_supervisory_2014,ahn_experimental_2015,ahn_milp_2016}. In \cite{colombo_efficient_2012,colombo_least_2014,bruni_robust_2013,ahn_supervisory_2014,ahn_experimental_2015}, the safety verification problem is solved exactly with an assumption that the paths of vehicles intersect only at a single conflict point. Multiple conflict points are considered in \cite{ahn_milp_2016} and the safety verification problem is solved exactly when vehicle dynamics are restricted to first-order linear dynamics. By contrast, the main contribution of this paper is to solve the safety verification problem on multiple conflict points for general longitudinal vehicle dynamics, which are nonlinear and second-order. A similar problem of robots following predefined paths is considered in \cite{peng_convexity_2005,peng_coordinating_2005}, but their approach is not designed for safety verification and thus restricted to zero initial speed with the double integrator dynamics. Our scheduling approach can deal with general vehicle dynamics and verify safety at any given state.

Recently, intersection management has been receiving considerable research attention. Most of the recent works concentrate on autonomous intersection management, where a controller takes control of vehicles at all times until they cross the intersection \cite{kim_mpc-based_2014,kamal_vehicle-intersection_2014,zhu_linear_2015,murgovski_convex_2015,chen_cooperative_2015,tachet_revisiting_2016,altche_timeoptimal_2016}. Our approach, instead, is to design a least restrictive supervisor in the sense that it overrides drivers only when they cannot avoid a collision. 

Collision avoidance for multiple vehicles has been an active area of research mostly in air traffic management. Various approximation approaches are employed to solve collision avoidance problems, such as approximation of dynamics \cite{richards_spacecraft_2002,pallottino_conflict_2002,vela_near_2010,alonso-ayuso_collision_2011,alonso-ayuso_modeling_2013,omer_space-discretized_2015} or relaxation of the original problems \cite{soler_hybrid_2016,chen_three-dimensional_2016}. The controllers presented in these works are not least restrictive, as opposed to the controller considered in this paper. While least restrictive controllers are presented in  \cite{tomlin_conflict_1998,tomlin_game_2000}, they are applicable only to a small number of vehicles due to the computational complexity of safety verification. As a scalable approach with the number of vehicles, decentralized control is also employed in \cite{pallottino_decentralized_2007,keviczky_decentralized_2008,zhang_hierarchical_2012}. However, decentralized control usually terminates with suboptimal solutions or deadlock. In this paper, we present the design of a centralized controller and prove that it is non-blocking. We validate through computer simulations that this controller can run in real time for realistic size scenarios such as that illustrated in Figure~\ref{figure:general_intersection}.

%
%
%

The rest of this paper is organized as follows. In Section~\ref{section:system}, we define an intersection model and a vehicle dynamic model. The safety verification problem is stated in Section~\ref{section:problemStatement} and translated into a scheduling problem in Section~\ref{section:excatSolutions}. To solve this problem, we formulate the lower and upper bound problems and quantify their approximation bounds in Section~\ref{section:approximateSolutions}. Based on the upper bound problem, a supervisory algorithm is introduced and proved to be non-blocking in Section~\ref{section:supervisor}. We present simulation results in Section~\ref{section:simulation} and conclude the paper in Section~\ref{section:concluisions}.

\section{System Definition}\label{section:system}

\subsection{Intersection Model}
\begin{figure}[t!]
\centering
\includegraphics[width=.7\columnwidth]{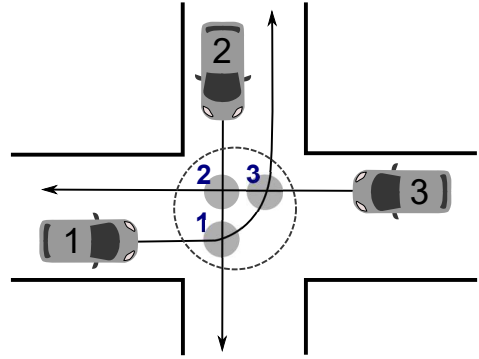}
\caption{An intersection is modeled as a set of conflict areas near which two longitudinal predefined paths intersect. }
\label{figure:intersection}
\end{figure}
At a road intersection, vehicles tend to follow predefined paths, which intersect at several conflict points. We define an area around each conflict point accounting for the length of vehicles and call it a conflict area. In this paper, we model an intersection as a collection of all conflict areas. For example, in Figure~\ref{figure:intersection}, the intersection is modeled as a set of conflict areas 1-3. 

The main focus of this paper is to prevent collisions among vehicles whose paths intersect at conflict areas. Thus, we assume that there is only one vehicle per lane and neglect rear-end collisions. These can be included using a similar approach as used in \cite{colombo_least_2014}.

\subsection{Vehicle Dynamical Model}
With a vehicle state $(x_j, \dot{x}_j)$ where $x_j\in X_j\subseteq \mathbb{R}$ is the position of vehicle $j$ on its longitudinal path and $\dot{x}_j\in \dot{X}_j:=[\dot{x}_{j,min}, \dot{x}_{j,max}]\subset \mathbb{R}$ is the speed, the longitudinal dynamics of vehicle $j$ are described as follows:
\begin{align}\label{equation:model}
\ddot{x}_j = f_j(x_j, \dot{x}_j, u_j).
\end{align}
The input $u_j$ is the throttle or brake input in the space $U_j := [u_{j,min}, u_{j,max}]\subset \mathbb{R}$.

Let us consider $n$ vehicles approaching an intersection. The whole system dynamics can be obtained by combining the individual dynamics \eqref{equation:model} and written as follows:
\begin{align}\label{equation:whole_model}
\ddot{\mathbf x} = \mathbf{f(x, \dot{x}, u)},
\end{align}
where $\mathbf{x} = (x_1,\ldots,x_n) \in \mathbf{X}\subseteq \mathbb{R}^n$ and similarly, $\dot{\mathbf{x}}\in\dot{\mathbf{X}}=[\dot{\mathbf{x}}_{min},\dot{\mathbf{x}}_{max}]\subset \mathbb{R}^n$, $\mathbf{u}\in \mathbf{U}=[\mathbf{u}_{min}, \mathbf{u}_{max}]\subset \mathbb{R}^n$.

We define an input signal $u_j(\cdot):t\in\mathbb{R}\mapsto u_j(t)\in U_j$ in the input signal space $\mathcal{U}_j$. Let $x_j(t,u_j(\cdot),x_j(0),\dot{x}_j(0))$ denote the position reached at time $t$ starting from $(x_j(0),\dot{x}_j(0))$ using an input signal $u_j(\cdot)\in\mathcal{U}_j$. We also use the aggregate position $\mathbf{x}(t,\mathbf{u}(\cdot), \mathbf{x}(0), \dot{\mathbf{x}}(0))$ with the aggregate input signal $\mathbf{u}(\cdot)\in \mathbf{\mathcal{U}}$. Similarly, we use $\dot{\mathbf{x}}(t,\mathbf{u}(\cdot), \mathbf{x}(0), \dot{\mathbf{x}}(0))$ to denote the speed at time $t$ evolving with $\mathbf{u}(\cdot)$. Regarding these, we make the following assumption. 

\begin{assumption}\label{assumption:path-connected}
For all $j\in\{1,\ldots,n\}$, the position $x_j(t,u_j(\cdot),x_j(0),\dot{x}_j(0))$ depends continuously on $u_j(\cdot)\in\mathcal{U}_j$, and the input signal space $\mathcal{U}_j$ is path-connected.
\end{assumption}
We say $u_j(\cdot)\leq u_j'(\cdot)\in\mathcal{U}_j$ if $u_j(t)\leq u_j'(t)\in U_j$ for all $t\geq 0$. We consider $\dot{x}_{j,min} > 0$ to exclude a trivial scenario in which vehicles come to a full stop before an intersection and do not cross it. Most importantly, we assume that the individual dynamics \eqref{equation:model} are monotone, that is, they satisfy the following property.

\begin{assumption}\label{assumption:order-preserving}
	For all $j\in\{1,\ldots,n\}$, if $u_j(\cdot)\leq u'_j(\cdot)$, $x_j(0)\leq x'_j(0)$, $\dot{x}_j(0)\leq \dot{x}'_j(0)$, and $t\leq t'$, $$x_j(t,u_j(\cdot),x_j(0),\dot{x}_j(0))\leq x_j(t',u'_j(\cdot),x'_j(0),\dot{x}'_j(0))$$ for all $t\geq 0$.
\end{assumption}

\section{Problem Statement}\label{section:problemStatement}
Let us consider $n$ vehicles approaching an intersection that is modeled as a collection of $m$ conflict areas. We denote the location of conflict area $i$ on the longitudinal path of vehicle $j$ as an open interval $(\alpha_{ij}, \beta_{ij})\subset \mathbb{R}$. We say a collision occurs at an intersection if two vehicles stay inside the same conflict area simultaneously. This configuration is referred to as a \textit{bad set}, which is denoted by $\mathcal{B}$ and defined as follows:
\begin{align}
\begin{split}\label{equation:badset}
\mathcal{B}&:=\{\mathbf{x}\in X: x_j\in (\alpha_{ij},\beta_{ij})~\text{and}~x_{j'}\in (\alpha_{ij'}, \beta_{ij'})\\
&\text{for some}~i\in\{1,\ldots,m\}~\text{and}~j\ne j'\in\{1,\ldots,n\}\}
\end{split}
\end{align}

The main interest of this paper is safety verification, that is, verifying whether the system can avoid entering the bad set at all future time. We approach this by stating a mathematical problem, called the \textit{safety verification} problem, as follows.

\begin{problem}[safety verification]\label{problem:verification}
	Given an initial condition $(\mathbf{x}(0),\dot{\mathbf{x}}(0))$, determine if there exists $\mathbf{u}(\cdot)\in\mathcal{U}$ such that $\mathbf{x}(t,\mathbf{u}(\cdot), \mathbf{x}(0),\dot{\mathbf{x}}(0))\notin \mathcal{B}$ for all $t\geq 0$.
\end{problem}

To answer this problem, we need to evaluate all possible input signals, which are functions of time, until finding one satisfying the condition. To avoid this exhaustive and infinite set of computations, we translate this problem to a scheduling problem, which is to find feasible schedules, non-negative real numbers, for vehicles to cross the conflict areas. The rationale behind this translation is that real numbers are computationally less complicated to manipulate than functions. While this scheduling problem is still not easy to solve, we can provide its approximate solutions efficiently.

\section{Scheduling: Equivalent Problem to Problem~\ref{problem:verification}}\label{section:excatSolutions}
In this section, we formulate a scheduling problem and present the theorem stating that this problem is equivalent to the safety verification problem (Problem~\ref{problem:verification}).  

\begin{figure}[t!]
\centering
\subfloat[]{	
\includegraphics[width=0.4\columnwidth]{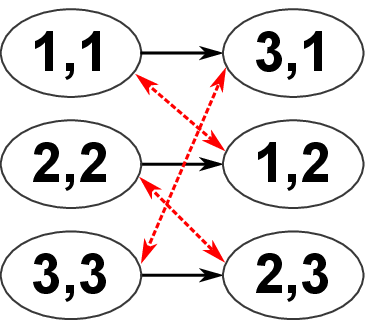}}
\quad
\subfloat[]{	
\includegraphics[width=0.4\columnwidth]{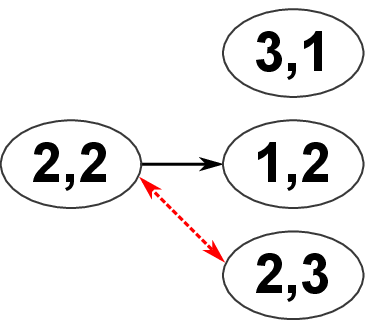}}
\caption{(a) All operations of the scenario in Figure~\ref{figure:intersection}. (b) Operations when $\beta_{11}\leq x_1(0) < \beta_{31}, x_2(0) <\beta_{22},$ and $\beta_{33}\leq x_3(0) < \beta_{23}$. See Example~\ref{example:opertions} for more details.}
\label{figure:operations}
\end{figure}

A scheduling problem can be described by a graph representation \cite{pinedo_scheduling_2012}. For a node, let $(i,j)$ be an operation of vehicle $j$ processing on conflict area $i$. The collection of all operations is denoted by $\bar{\mathcal N}$. 
\begin{align*}
\bar{\mathcal{N}}:=\{(i,j)&\in\{1,\ldots,m\}\times\{1,\ldots,n\}: \\
&\text{conflict area}~i~
\text{is on the route of vehicle}~j\}.
\end{align*}
We also define a set $\mathcal{N}\subseteq \bar{\mathcal{N}}$ that contains operations to be processed given an initial position $\mathbf{x}(0)$. 
\begin{align*}
{\mathcal{N}}:=\{(i,j)\in\bar{\mathcal N}: x_j(0)< \beta_{ij}\}.
\end{align*}
Recall that $(\alpha_{ij}, \beta_{ij})\subset \mathbb{R}$ denotes the location of conflict area $i$ on the longitudinal path of vehicle $j$.
Notice that $\mathcal{N}$ is the set of all the operations of interest given $\mathbf{x}(0)$ because $\beta_{ij}\leq x_j(0)$ indicates that vehicle $j$ has exited conflict area $i$. 

A first operation set $\mathcal{F}\subseteq\mathcal N$ and a last operation set $\mathcal{L}\subseteq\mathcal{N}$ are defined as follows: 
\begin{align*}
{\mathcal{F}}:=\{(i,j)\in{\mathcal{N}}: ~&(i,j)~\text{is the first operation of vehicle $j$}\},\\
{\mathcal{L}}:=\{(i,j)\in{\mathcal{N}}: ~&(i,j)~\text{is the last operation of vehicle $j$}\}.
\end{align*}

Now, let us define arcs in the graph. We define sets of conjunctive arcs $\mathcal{C}$ and disjunctive arcs $\mathcal D$, which connect two operations in $\mathcal{N}$ as follows:
\begin{align*}
&\mathcal{C}:=\{(i,j)\rightarrow (i',j): \text{vehicle}~j~\text{crosses conflict area}~i\\
&\hspace{0.3 in}\text{and then conflict area}~i'~\text{for some}~i,j,j'\},\\
&\mathcal{D}:=\{(i,j)\leftrightarrow (i,j'): ~\text{vehicles $j$ and $j'$ share}\\
&\hspace{0.3 in}\text{the same conflict area $i$}~\text{for some}~i,j,j'\}.
\end{align*}
That is, an element in $\mathcal{C}$ represents the sequence of operations on the path of a vehicle, and an element in $\mathcal{D}$ represents the undetermined sequence of operations on the same conflict area. 

\begin{example}\label{example:opertions}
In the scenario in Figure~\ref{figure:intersection}, suppose $\beta_{11}\leq x_1(0) < \beta_{31}, x_2(0) <\beta_{22},$ and $\beta_{33}\leq x_3(0) < \beta_{23}$. The operations are illustrated in Figure~\ref{figure:operations}, where the operation sets $\bar{\mathcal{N}}$ and $\mathcal{N}$ become as follows:
\begin{align*}
&\bar{\mathcal N} = \{(1,1),(3,1),(2,2),(1,2),(3,3),(2,3)\},\\
&{\mathcal N} = \{(3,1),(2,2),(1,2),(2,3)\}.
\end{align*}
Here, the first and last operation sets are ${\mathcal F} = \{(3,1),(2,2),(2,3)\}, {\mathcal L} = \{(3,1),(1,2),(2,3)\}.$ The sets of conjunctive and disjunctive arcs are ${\mathcal C} = \{(2,2)\rightarrow (1,2)\}$ and ${\mathcal D} = \{(2,2)\leftrightarrow (2,3)\}$.
\end{example}

We now introduce scheduling parameters, release times, deadlines, and process times, to formulate a jobshop scheduling problem \cite{pinedo_scheduling_2012}.

Given an initial condition $(x_j(0), \dot{x}_j(0))$, let $T_{ij}$ be the time that vehicle $j$ will enter conflict area $i$, that is, $x_j(T_{ij},u_j(\cdot),x_j(0),\dot{x}_j(0))=\alpha_{ij}$ for some $u_j(\cdot)\in\mathcal{U}_j$. Let $\mathbf{T}$ be the set of $T_{ij}$ for all $(i,j)\in\mathcal{N}$. A jobshop scheduling problem is to find this set, called a schedule, such that vehicles never meet inside a conflict area.

Given $\mathbf{T}$, the release time $R_{ij}(\mathbf{T})$ is the soonest time at which vehicle $j$ can enter conflict area $i$ under the constraint that it enters the previous conflict area $i'$ at $T_{i'j}$. The deadline $D_{ij}(\mathbf{T})$ is the latest such time. The process time $P_{ij}(\mathbf{T})$ is the minimum time that vehicle $j$ takes to exit conflict area $i$ under the constraint that it enters the same conflict area at time $T_{ij}$ and the next conflict area $i''$ at time $T_{i''j}$. We omit the argument $\mathbf{T}$ if it is clear from context.

Formally, the release time and deadline are defined as follows. Given an initial condition $(\mathbf{x}(0),\dot{\mathbf{x}}(0))$ and $\mathbf{T}$, for all $(i,j)\in \mathcal{N}\setminus\mathcal{F}$, there is a preceding operation $(i',j)$ such that $(i',j)\rightarrow (i,j)\in\mathcal{C}$. 
	\begin{align}\label{definition:RD_nofirst}
	\begin{split}
	&R_{ij}(\mathbf{T}):=\min_{u_j(\cdot)\in\mathcal{U}_j} \{t:x_j(t,u_j(\cdot),x_j(0),\dot{x}_j(0))=\alpha_{ij}\\
	&\text{with constraint}~x_j(T_{i'j}, u_j(\cdot), x_j(0),\dot{x}_j(0))=\alpha_{i'j}\},\\
	&D_{ij}(\mathbf{T}):=\max_{u_j(\cdot)\in\mathcal{U}_j} \{t:x_j(t,u_j(\cdot),x_j(0),\dot{x}_j(0))=\alpha_{ij}\\
	&\text{with constraint}~x_j(T_{i'j}, u_j(\cdot), x_j(0),\dot{x}_j(0))=\alpha_{i'j}\}.
	\end{split}
	\end{align}
If the constraint is not satisfied, set $R_{ij}=\infty$ and $D_{ij}=-\infty$.
For all $(i,j)\in \mathcal{F}$, such a preceding operation $(i',j)$ does not exists. If $x_j(0)<\alpha_{ij}$,
	\begin{align}\label{definition:RD_first}
	\begin{split}
	&R_{ij}:=\min_{u_j(\cdot)\in\mathcal{U}_j} \{t:x_j(t,u_j(\cdot),x_j(0),\dot{x}_j(0))=\alpha_{ij}\},\\
	&D_{ij}:=\max_{u_j(\cdot)\in\mathcal{U}_j} \{t:x_j(t,u_j(\cdot),x_j(0),\dot{x}_j(0))=\alpha_{ij}\}.
	\end{split}
	\end{align}
If $\alpha_{ij}\leq x_j(0)$, set $R_{ij}=D_{ij}=0$. Notice that for $(i,j)\in\mathcal{F}$, the release time and deadline are independent of $\mathbf{T}$.

The process time is defined as follows. Given an initial condition $(\mathbf{x}(0),\dot{\mathbf{x}}(0))$ and $\mathbf{T}$, for all $(i,j)\in\mathcal{N}\setminus \mathcal{L}$, there is a succeeding operation $(i'',j)$ such that $(i,j)\rightarrow(i'',j)\in\mathcal{C}$. If $x_j(0)<\alpha_{ij}$,
\begin{align}
	\begin{split}\label{definition:p_nolast}
	&P_{ij}(\mathbf{T}):=\min_{u_j(\cdot)\in\mathcal{U}_j}\{t:x_j(t,u_j(\cdot),x_j(0),\dot{x}_j(0))=\beta_{ij}\\
	&\text{with constraints}~x_j(T_{ij}, u_j(\cdot), x_j(0),\dot{x}_j(0))=\alpha_{ij}\\
	&\hspace{1.8 cm}\text{and}~x_j(T_{i''j}, u_j(\cdot), x_j(0),\dot{x}_j(0))=\alpha_{i''j}\}.
	\end{split}
\end{align}
Set $P_{ij}(\mathbf{T}):=\min_{u_j(\cdot)\in\mathcal{U}_j}\{t:x_j(t,u_j(\cdot),x_j(0),\dot{x}_j(0))=\beta_{ij}$ with constraint $x_j(T_{i''j},u_j(\cdot),x_j(0),\dot{x}_j(0))=\alpha_{i''j}\}$ if $\alpha_{ij}\leq x_j(0)<\beta_{ij}$.
For all $(i,j)\in \mathcal{L}$, such a succeeding operation $(i'',j)$ does not exist. If $x_j(0)<\alpha_{ij}$,
	\begin{align}
	\begin{split}\label{definition:p_last}
	&P_{ij}(\mathbf{T}):=\min_{u_j(\cdot)\in\mathcal{U}_j}\{t:x_j(t,u_j(\cdot),x_j(0),\dot{x}_j(0))=\beta_{ij}\\
	&\text{with constraint}~x_j(T_{ij}, u_j(\cdot), x_j(0),\dot{x}_j(0))=\alpha_{ij}\}.
	\end{split}\end{align}
If $\alpha_{ij}\leq x_j(0)<\beta_{ij}$, set $P_{ij}:=\min_{u_j(\cdot)\in\mathcal{U}_j} \{t:x_j(t,u_j(\cdot),x_j(0),\dot{x}_j(0))=\beta_{ij}\}$. If $\beta_{ij}\leq x_j(0)$, operation $(i,j)$ is not of interest since vehicle $j$ has already crossed conflict area $i$. If the constraints are not satisfied, set $P_{ij}=\infty$.

Using the definitions above, a jobshop scheduling problem is formulated as follows.
\begin{problem}[jobshop scheduling problem]\label{problem:MINLP}
Given an initial condition $(\mathbf{x}(0),\dot{\mathbf{x}}(0))$, determine if $s^{*} = 0$:
\begin{align*}
s^*:= \underset{\mathbf{T,k}}{\text{minimize}}~\max_{(i,j)\in\mathcal{N}}(T_{ij} - D_{ij}(\mathbf{T}),0)
\end{align*}
subject to 
\begin{align}
&\text{for all}~(i,j)\in \mathcal{N}, \hspace{1 cm} R_{ij}(\mathbf{T})\leq T_{ij}, \tag{P\ref{problem:MINLP}.1}\label{minlp:RD}\\
&\text{for all}~(i,j)\leftrightarrow (i,j')\in\mathcal{D}, \notag\\
&\hspace{1 cm}\begin{cases}\tag{P\ref{problem:MINLP}.2}\label{minlp:disjunctive}
P_{ij}(\mathbf{T})\leq T_{ij'}+M(1-k_{ijj'}),\\
P_{ij'}(\mathbf{T})\leq T_{ij} + M(1-k_{ij'j}),\\
k_{ijj'}+k_{ij'j} = 1.
\end{cases}
\end{align}
where $\mathbf{T}=\{T_{ij}:(i,j)\in\mathcal{N}\}$, $\mathbf{k}=\{k_{ijj'}\in \{0,1\}: \text{for all}~(i,j)\leftrightarrow (i,j')\in\mathcal{D}\}$, and $M>0$ is a large number.
\end{problem}

In the scheduling literature, $\max (T_{ij}-D_{ij}(\mathbf{T}),0)$ is called the maximum lateness. This cost indicates the existence of a schedule that violates the deadline, and thus its minimization is one of the most studied scheduling problems \cite{pinedo_scheduling_2012}.

 If $s^*=0$, we have $T_{ij}$ that satisfies $T_{ij}\leq D_{ij}(\mathbf{T})$ for all $(i,j)\in\mathcal{N}$. The fact that $T_{ij}$ is bounded by $R_{ij}(\mathbf{T})$ and $D_{ij}(\mathbf{T})$ encodes the bounds of the input. Constraint~\eqref{minlp:disjunctive} says that for two vehicles $j$ and $j'$ that share the same conflict area $i$, either vehicle $j'$ enters it after vehicle $j$ exits (when $k_{ijj'} = 1$) or the other way around (when $k_{ij'j}=1$). By this constraint, each conflict area is exclusively used by one vehicle at a time. Thus, the existence of such $\mathbf{T}$ and $\mathbf{k}$ that yield $s^*=0$ is equivalent to the existence of $\mathbf{u}(\cdot)\in \mathcal{U}$ to avoid the bad set. This is the essence of the proof of the following theorem. 

\begin{theorem}[\cite{ahn_milp_2016}]\label{theorem:equivalence}
Problem~\ref{problem:verification} is equivalent to Problem~\ref{problem:MINLP}.
\end{theorem}
In \cite{ahn_milp_2016}, Problem~\ref{problem:MINLP} was introduced as a feasibility problem and the above theorem was proved. 

Theorem~\ref{theorem:equivalence} implies the following: given an initial condition $(\mathbf{x}(0), \dot{\mathbf{x}}(0))$,
\begin{align*}
	&s^* = 0 \implies \text{a safe input signal exists to avoid}~\mathcal{B},\\
	&s^* >0 \implies \text{no safe input signal exists to avoid}~\mathcal{B}.
\end{align*}
That is, $s^*$ is the indicator of the vehicles' safety.

While this theorem holds for general dynamics \eqref{equation:model}, Problem~\ref{problem:MINLP} can be difficult to solve depending on which vehicle dynamics are considered. In \cite{ahn_milp_2016}, vehicle dynamics are assumed to be first-order and linear in which case Problem~\ref{problem:MINLP} becomes a Mixed Integer Linear Programming (MILP) problem, which can be easily solved by a commercially available solver, such as CPLEX \cite{cplex_2015}. In this paper, due to the nonlinear and higher order vehicle dynamics, the constraints are nonlinear in $T_{ij}$ and therefore, Problem~\ref{problem:MINLP} is a Mixed Integer Non-Linear Programming (MINLP) problem, which is notorious for its computational intractability. To approximately solve Problem~\ref{problem:MINLP}, we formulate two MILP problems that yield lower and upper bounds of $s^*$, respectively. The MILP problem that computes the lower bound is a reformulation of the MILP problem given in \cite{ahn_milp_2016}, and the MILP problem that computes the upper bound is based on nonlinear second-order dynamics with a limited input space.


\section{Approximate Solutions to Problem~\ref{problem:MINLP}}\label{section:approximateSolutions}
In this section, we provide two MILP problems that yield lower and upper bounds of the optimal cost of Problem~\ref{problem:MINLP}. Using these bounds, we can quantify the approximation error between the approximate solution and the exact solution to Problem~\ref{problem:MINLP}.

\subsection{Lower bound problem}\label{section:lower}

Let us consider first-order vehicle dynamics, that is,  $$\dot{{\chi}}=\mathbf{v},$$ where $\chi\in \mathbf{X}$ is the vector representing the position of vehicles on their longitudinal paths, and $\mathbf{v}$ is the input. The input $\mathbf{v}$ lies in the space $\dot{\mathbf{X}}=[\dot{\mathbf{x}}_{min}, \dot{\mathbf x}_{max}]$ with $\dot{\mathbf{x}}_{min}> \mathbf{0}$. We also define an input signal ${v}_j(\cdot): \mathbb{R}_+\rightarrow  \dot{X}_j$ in the space $\mathcal{V}_j$ for vehicle $j$.

We define release times and deadlines for the lower bound problem. Process times are considered as decision variables.
\begin{definition}\label{definition:low_RD}
Given an initial condition $(\mathbf{x}(0), \dot{\mathbf{x}}(0))$ and $\chi(0) = \mathbf{x}(0)$, release times $r_{ij}$ and deadlines $d_{ij}$ are defined as follows. 

For all $(i,j)\in\mathcal{F}$, if $x_j(0)<\alpha_{ij}$,
\begin{align*}
&r_{ij} := \min_{u_j(\cdot)\in\mathcal{U}_j}\{t:x_j(t,u_j(\cdot),x_j(0), \dot{x}_j(0))=\alpha_{ij}\},\\
&d_{ij} := \max_{u_j(\cdot)\in\mathcal{U}_j}\{t:x_j(t,u_j(\cdot),x_j(0), \dot{x}_j(0))=\alpha_{ij}\}.
\end{align*}
If $\alpha_{ij}\leq x_j(0)$, then $r_{ij} = d_{ij}=0$.

For $(i,j)\in\mathcal{N}\setminus \mathcal{F}$, there exists a preceding operation $(i',j)$ such that $(i',j)\rightarrow (i,j)\in\mathcal{C}$. Given $p_{i'j}\geq 0$ such that  $\chi_j(p_{i'j},v_j(\cdot),\chi_j(0))=\beta_{i'j}$ for some $v_j(\cdot)\in\mathcal{V}_j$,
\begin{align}\label{definition:low_RD_NF}
&r_{ij} := p_{i'j}+\frac{\alpha_{ij}-\beta_{i'j}}{\dot{x}_{j,max}},&d_{ij} := p_{i'j}+\frac{\alpha_{ij}-\beta_{i'j}}{\dot{x}_{j,min}}.
\end{align}
\end{definition}
Notice that for the first operations, $(i,j)\in\mathcal{F}$, release times and deadlines consider general dynamics \eqref{equation:model} and thus $r_{ij} = R_{ij}$ and $d_{ij}=D_{ij}$ by \eqref{definition:RD_first}. This is a different definition from that in \cite{ahn_milp_2016} and results in a tighter constraint than $r_{ij} = (\alpha_{ij}-\chi_j(0))/\dot{x}_{j,max}$ and $d_{ij} = (\alpha_{ij}-\chi_j(0))/\dot{x}_{j,min}$.

Using Definition~\ref{definition:low_RD}, the lower bound problem is formulated as a decision problem in which the maximum lateness is minimized. 

\begin{problem}[lower bound problem]\label{problem:milp_first}
Given an initial condition $(\mathbf{x}(0),\dot{\mathbf{x}}(0))$, determine if $s^{*}_L = 0$:
\begin{align*}
s_{L}^*:= \underset{\mathbf{t,p,k}}{\text{minimize}}~\max_{(i,j)\in\mathcal{N}}(t_{ij} - d_{ij},0)
\end{align*}
subject to 
\begin{align}
&\text{for all}~(i,j)\in \mathcal{N},\hspace{0.5 cm} r_{ij}\leq t_{ij},  \tag{P\ref{problem:milp_first}.1}\label{milp1:RD}\\
&\text{for all}~(i,j)\in\mathcal{N},\hspace{0.5 cm} \frac{\beta_{ij}-\alpha_{ij}}{\dot{x}_{j,max}}\leq p_{ij}-t_{ij} \leq \frac{\beta_{ij}-\alpha_{ij}}{\dot{x}_{j,min}},\tag{P\ref{problem:milp_first}.2}\label{milp1:process}\\
&\text{for all}~(i,j)\leftrightarrow (i,j')\in\mathcal{D}, \notag\\
&\hspace{1 cm}\begin{cases}\tag{P\ref{problem:milp_first}.3}\label{milp1:disjunctive}
p_{ij}\leq t_{ij'}+M(1-k_{ijj'}),\\
p_{ij'}\leq t_{ij} + M(1-k_{ij'j}),\\
k_{ijj'}+k_{ij'j} = 1.
\end{cases}
\end{align}
where $\mathbf{t}=\{t_{ij}:(i,j)\in\mathcal{N}\}$, $\mathbf{p}=\{p_{ij}:(i,j)\in\mathcal{N}\}$, $\mathbf{k}:=\{k_{ijj'}\in \{0,1\}: \text{for all}~(i,j)\leftrightarrow (i,j')\in\mathcal{D}\}$ and $M$ is a large number in $\mathbb{R}_+$.
\end{problem}

From~\eqref{definition:low_RD_NF}, we know that the objective function and constraint~\eqref{milp1:RD} are linear with the decision variables. Thus, this problem is a Mixed Integer Linear Programming (MILP) problem.
\begin{theorem}\label{theorem:lower}
$s_L^* \leq s^*$
\end{theorem}
\begin{proof}
Suppose Problem~\ref{problem:MINLP} finds $\mathbf{T}^*$ and $\mathbf{k}^*$ with the corresponding cost $s^*$, whether or not $s^*=0$. We will show that $\tilde{\mathbf{t}} = \mathbf{T}^*$ and $\tilde{\mathbf{k}} = \mathbf{k}^*$ become a feasible solution for Problem~\ref{problem:milp_first} with some $\tilde{\mathbf{p}}$. 

Given $\mathbf{T}^*$, we have $P_{ij}(\mathbf{T}^*)$ for all $(i,j)\in\mathcal{N}$. Consider $\tilde{p}_{ij} = P_{ij}(\mathbf{T}^*)$. From \eqref{definition:p_nolast} and \eqref{definition:p_last}, $P_{ij}(\mathbf{T}^*) - T_{ij}^* = \tilde{p}_{ij} - \tilde{t}_{ij}$ is the time to reach $\beta_{ij}$ from $\alpha_{ij}$ and thus satisfies \eqref{milp1:process}. Constraint~\eqref{milp1:disjunctive} is the same as constraint~\eqref{minlp:disjunctive} in Problem~\ref{problem:MINLP}.

We will show that $r_{ij}\leq R_{ij}(\mathbf{T}^*)$ and $D_{ij}(\mathbf{T}^*)\leq d_{ij}$. As mentioned earlier, for $(i,j)\in\mathcal{F}$, $r_{ij}=R_{ij}$ and $d_{ij} = D_{ij}$. For $(i,j)\in\mathcal{N}\setminus \mathcal{F}$, we have a preceding operation $(i',j)$ such that $(i',j)\rightarrow(i,j)\in\mathcal{C}$. By \eqref{definition:RD_nofirst}, $R_{ij}$ is equal to $T^*_{i'j}$ plus the minimum time to reach $\alpha_{ij}$ from $\alpha_{i'j}$. Considering this definition with \eqref{definition:p_nolast}, $R_{ij}$ is again equal to $P_{i'j}$ plus the minimum time to reach $\alpha_{ij}$ from $\beta_{i'j}$, thereby $R_{ij}\geq P_{i'j} + (\alpha_{ij}-\beta_{i'j})/\dot{x}_{j,max}$. Also, since $P_{i'j}= \tilde{p}_{i'j}$, we have $P_{i'j} + (\alpha_{ij}-\beta_{i'j})/\dot{x}_{j,max} = r_{ij}$. Thus $r_{ij}\leq R_{ij}$. Similarly, $D_{ij}\leq d_{ij}$. By these inequalities, $R_{ij}\leq T^*_{ij}$ implies $r_{ij}\leq T^*_{ij}=\tilde{t}_{ij}$, which is constraint~\eqref{milp1:RD}.

 Therefore, $\tilde{\mathbf t},\tilde{\mathbf p},$ and $\tilde{\mathbf k}$ is a feasible solution of Problem~\ref{problem:milp_first}. Since $D_{ij}\leq d_{ij}$, $s_L^*\leq \max(\tilde{t}_{ij}-d_{ij},0) \leq \max(T^*_{ij}-D_{ij},0)= s^*$.
\end{proof}

\subsection{Upper bound problem}\label{section:upper}

In this section, we relax Problem~\ref{problem:MINLP} to a MILP problem by considering general dynamics \eqref{equation:model} on a restricted input space. This problem is to find a set of times at which vehicles enter their \textit{first conflict area}, assuming that to reach the following conflict areas all vehicles apply maximum input. The rationale here is that once a vehicle enters an intersection, the driver tries to exit as soon as possible.

We define $\alpha_{j,min}$ to denote the first conflict area as follows:
$$\alpha_{j,min}:=\min_{(i,j)\in\bar{\mathcal N}} \alpha_{ij}.$$
Recall that $\bar{\mathcal N}$ is a set of all operations independent of an initial condition.

In the upper bound problem, the time to reach the first conflict area is a decision variable, as opposed to Problems~\ref{problem:MINLP} and \ref{problem:milp_first} whose decision variables are the entering times for all conflict areas. This decision variable, called a schedule, is denoted by $\mathbf{T}^{\mathcal F}=\{T_{ij}^\mathcal{F}:\text{for all}~(i,j)\in{\mathcal{F}}\}$ and defined as $x_j(T_{ij}^\mathcal{F},u_j(\cdot),x_j(0),\dot{x}_j(0))=\alpha_{ij}~\text{for some}~ u_j(\cdot)\in\mathcal{U}_j$ if $x_j(0)<\alpha_{ij}$ and otherwise $T_{ij}^\mathcal{F}=0$. 

The release times and deadlines are defined only for the first operation as follows.

\begin{figure}[t!]
\centering
\subfloat[$x_j(0)<\alpha_{j,min}=\alpha_{i'j}$]{	
\includegraphics[width = 0.46\columnwidth]{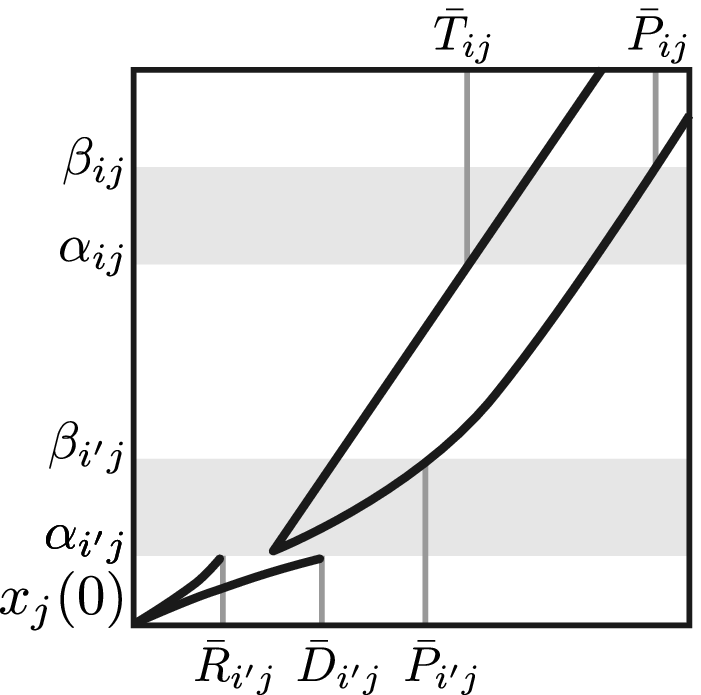}}
\quad
\subfloat[$\alpha_{i'j}\leq x_j(0)<\beta_{i'j}$]{	
\includegraphics[width = 0.46 \columnwidth]{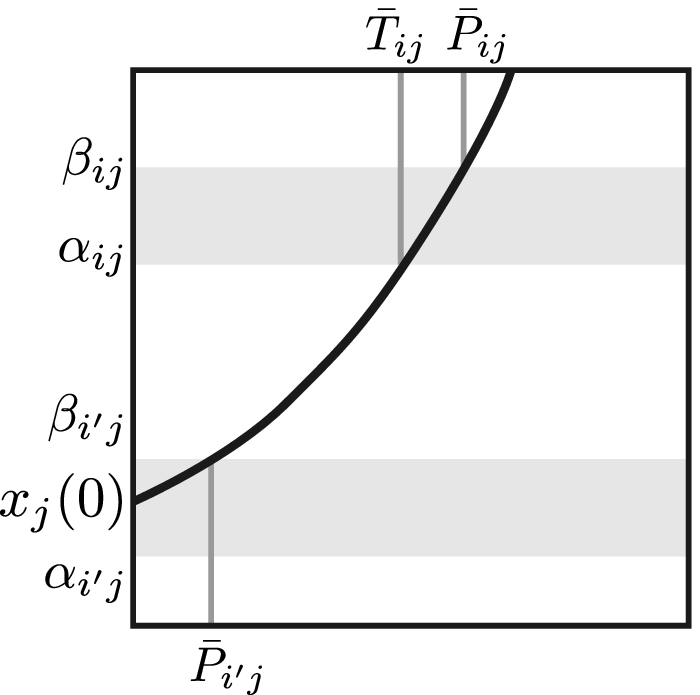}}
\caption{Illustration of Definitions~\ref{definition:up_RD} and \ref{definition:up_TP}. Suppose $(i',j)\in{\mathcal{F}}$ and $(i',j)\rightarrow (i,j)\in\mathcal{C}$. We can compute $\bar{P}_{i'j}, \bar{T}_{ij}, \bar{P}_{ij}$ by considering the maximum input inside the intersection.}
\label{figure:approach1_definitions}
\end{figure}

\begin{definition}\label{definition:up_RD}
Given an initial condition $(\mathbf{x}(0),\dot{\mathbf x}(0))$, release times $\bar{\mathbf{R}}=\{\bar{R}_{ij}:\text{for all}~(i,j)\in {\mathcal{F}}\}$ and deadlines $\bar{\mathbf{D}}=\{\bar{D}_{ij}:\text{for all}~(i,j)\in {\mathcal{F}}\}$ are defined as follows.

For all $(i,j)\in{\mathcal{F}}$, if $x_j(0)<\alpha_{j,min}$,
\begin{align}
\begin{split}\label{definition:milp2_RD}
\bar{R}_{ij}:=\min_{u_j(\cdot)\in\mathcal{U}_j}\{t:x_j(t,u_j(\cdot),x_j(0),\dot{x}_j(0))=\alpha_{j,min}\},\\
\bar{D}_{ij}:=\max_{u_j(\cdot)\in\mathcal{U}_j}\{t:x_j(t,u_j(\cdot),x_j(0),\dot{x}_j(0))=\alpha_{j,min}\}.
\end{split}
\end{align}
If $\alpha_{j,min}\leq x_j(0)<\alpha_{ij}$, 
\begin{align}
\begin{split}\label{definition:milp2_RD_inside}
&\bar{R}_{ij}:=\min_{u_j(\cdot)\in\mathcal{U}_j}\{t:x_j(t,u_j(\cdot),x_j(0),\dot{x}_j(0))=\alpha_{ij}\},\\
&\bar{D}_{ij} = \bar{R}_{ij}.
\end{split}
\end{align}
If $\alpha_{ij}\leq x_j(0)$, set $\bar{R}_{ij}=\bar{D}_{ij}=0$.
\end{definition}

Notice that $\bar{R}_{ij}=R_{ij}$ and $\bar{D}_{ij}=D_{ij}$ if $x_j(0)<\alpha_{j,min}$, and $\bar{D}_{ij} = \bar{R}_{ij} = R_{ij}$ if $x_j(0)\geq \alpha_{j,min}$. The release time $\bar{R}_{ij}$ and the deadline $\bar{D}_{ij}$ depend only on an initial condition $(x_j(0), \dot{x}_j(0))$, not on the decision variable $T_{ij}^{\mathcal{F}}$. 

Given a schedule $\mathbf{T}^{\mathcal F}$, we define $\bar{\mathbf{T}}=\{\bar{T}_{ij}:\text{for all}~(i,j)\in{\mathcal{N}}\}$ and $\bar{\mathbf{P}}=\{\bar{P}_{ij}:\text{for all}~(i,j)\in{\mathcal{N}}\}$ as illustrated in Figure~\ref{figure:approach1_definitions}. Suppose vehicle $j$ has the first operation $(i',j)$. When vehicle $j$ is outside the intersection $(x_j(0)<\alpha_{j,min})$, $\bar{T}_{ij}$ and $\bar{P}_{ij}$ represent the minimum times at which vehicle $j$ can enter and exit conflict area $i$, respectively, \textit{no matter what speed} it has at $T_{i'j}^{\mathcal F}$. When vehicle $j$ is inside the intersection $(\alpha_{j,min}\leq x_j(0))$,  $\bar{T}_{ij}$ and $\bar{P}_{ij}$ are the minimum times at which it enters and exits conflict area $i$, respectively. These are formally defined as follows. 
\begin{definition}\label{definition:up_TP}
Given an initial condition $(\mathbf{x}(0),\dot{\mathbf x}(0))$ and a schedule $\mathbf{T}^{{\mathcal{F}}}$, we define $\bar{\mathbf{T}}=\{\bar{T}_{ij}:\text{for all}~(i,j)\in{\mathcal{N}}\}$ and $\bar{\mathbf{P}}=\{\bar{P}_{ij}:\text{for all}~(i,j)\in{\mathcal{N}}\}$ as follows.

If $x_{j}(0)<\alpha_{j,min}$, for $(i,j)\in {\mathcal{F}}$, 
	\begin{align}
	&\bar{T}_{ij} = T_{ij}^\mathcal{F},\label{definition:milp2_T_F}\\
	& \bar{P}_{ij} = T_{ij}^\mathcal{F}+\min_{u_j(\cdot)\in\mathcal{U}_j}\{t:x_j(t,u_j(\cdot),\alpha_{ij}, \dot{x}_{j,min})=\beta_{ij}\},\notag
	\end{align}
	and for $(i,j)\in\mathcal{N}\setminus{\mathcal{F}}$, there exists the first operation $(i',j)$ such that $(i',j)\in {\mathcal{F}}$ and $\alpha_{i'j}=\alpha_{j,min}$.
	\begin{align}
	\begin{split}\label{definition:milp2_T_NF}
	&\bar{T}_{ij} = T_{i'j}^\mathcal{F}+ \min_{u_j(\cdot)\in\mathcal{U}_j} \{t:x_j(t,u_j(\cdot),\alpha_{j,min},\dot{x}_{j,max})=\alpha_{ij}\},\\
	&\bar{P}_{ij} = T_{i'j}^\mathcal{F}+ \min_{u_j(\cdot)\in\mathcal{U}_j} \{t:x_j(t,u_j(\cdot),\alpha_{j,min},\dot{x}_{j,min})=\beta_{ij}\}.		
	\end{split}
	\end{align}	

	If $\alpha_{j,min}\leq x_j(0)$, for $(i,j)\in\mathcal{F}$,
	\begin{align}
		\begin{split}\label{definition:milp2_T_F_inside}
		&\bar{T}_{ij} = T_{ij}^{\mathcal F},\\
		&\bar{P}_{ij} = T_{ij}^{\mathcal F}-\bar{R}_{ij}+\\&\hspace{1 cm}\min_{u_j(\cdot)\in\mathcal{U}_j}\{t: x_j(t,u_j(\cdot),x_j(0),\dot{x}_j(0))=\beta_{ij}\}.
		\end{split}
	\end{align}
	and for $(i,j)\in\mathcal{N}\setminus\mathcal{F}$, there exists the first operation $(i',j)$ such that $(i',j)\in\mathcal{F}$.
	\begin{align}
		\begin{split}\label{definition:milp2_T_NF_inside}
		&\bar{T}_{ij} =  T_{i'j}^{\mathcal F}-\bar{R}_{i'j} + \\&\hspace{1 cm}\min_{u_j(\cdot)\in\mathcal{U}_j}\{t: x_j(t,u_j(\cdot),x_j(0),\dot{x}_j(0))=\alpha_{ij}\},\\
		&\bar{P}_{ij} = T_{i'j}^{\mathcal F}-\bar{R}_{i'j} +\\&\hspace{1 cm}\min_{u_j(\cdot)\in\mathcal{U}_j}\{t: x_j(t,u_j(\cdot),x_j(0),\dot{x}_j(0))=\beta_{ij}\}.
		\end{split}
	\end{align}
	In \eqref{definition:milp2_T_F_inside} and \eqref{definition:milp2_T_NF_inside}, we have $\bar{R}_{ij}=\bar{D}_{ij}$ by \eqref{definition:milp2_RD_inside} and thus for $(i',j)\in\mathcal{F}$, $T_{i'j}^{\mathcal F} - \bar{R}_{i'j}=0$ if $T_{i'j}^\mathcal{F}\in [\bar{R}_{i'j},\bar{D}_{i'j}]$.
\end{definition}

Using these definitions, we formulate the upper bound problem.

\begin{problem}[upper bound problem]\label{problem:milp_upper}
Given an initial condition $(\mathbf{x}(0),\dot{\mathbf{x}}(0))$, determine if $s_U^{*} = 0$:
\begin{align*}
s_U^*:= \underset{\mathbf{T^{\mathcal{F}},k}}{\text{minimize}}~\max_{(i,j)\in\mathcal{F}}({T}_{ij}^{\mathcal{F}} - \bar{D}_{ij},0)
\end{align*}
subject to 
\begin{align}
&\text{for all}~(i,j)\in {\mathcal{F}}, \hspace{1 cm} \bar{R}_{ij}\leq {T}^{\mathcal{F}}_{ij}, \tag{P\ref{problem:milp_upper}.1}\label{milp2:RD}\\
&\text{for all}~(i,j)\leftrightarrow (i,j')\in\mathcal{D}, \notag\\
&\hspace{1 cm}\begin{cases}\tag{P\ref{problem:milp_upper}.2}\label{milp2:disjunctive}
\bar{P}_{ij}\leq \bar{T}_{ij'}+M(1-k_{ijj'}),\\
\bar{P}_{ij'}\leq \bar{T}_{ij} + M(1-k_{ij'j}),\\
k_{ijj'}+k_{ij'j} = 1.
\end{cases}
\end{align}
where $\mathbf{T}^\mathcal{F}=\{T_{ij}^\mathcal{F}:\text{for all}~(i,j)\in {\mathcal{F}}\}$, $\mathbf{k}=\{k_{ijj'}\in \{0,1\}: \text{for all}~(i,j)\leftrightarrow (i,j')\in\mathcal{D}\}$, and $M$ is a large number in $\mathbb{R}_+$.
\end{problem}


The constraints in the problem are written in linear forms with the decision variable $\mathbf{T}^{\mathcal F}$ as noticed in \eqref{definition:milp2_T_F}-\eqref{definition:milp2_T_NF_inside}. Thus, the problem is a MILP problem. 

We will show that $s_U^*$ in Problem~\ref{problem:milp_upper} can be considered as an upper bound of $s^*$ in Problem~\ref{problem:MINLP} in the sense that $s^*\leq M s_U^*$  for a large number $M > 0$. This inequality is not trivial if $s_U^*=0$. In the following thereom, therefore, we will show that $s_U^*=0$ implies $s^*=0$ for any initial condition $(\mathbf{x}(0), \dot{\mathbf{x}}(0))$.

\begin{theorem}\label{theorem:upper}
$s_U^*=0 \Rightarrow s^*=0$.
\end{theorem}
\begin{proof}
Suppose Problem~\ref{problem:milp_upper} finds an optimal solution $\mathbf{T}^{\mathcal{F}*}=\{T_{ij}^{\mathcal{F}*}: (i,j)\in\mathcal{F}\}$ and $\mathbf{k}^* = \{k^*_{ijj'}: (i,j)\leftrightarrow (i,j')\in\mathcal{D}\}$ that yields $s_U^*=0$.

We define $\tilde{\mathbf T}$ and $\tilde{\mathbf P}$ as follows. For $(i,j)\in\mathcal{F}$, if $x_j(0) < \alpha_{ij}$, 
\begin{align}\label{proof:upper_tildeT_F}
\begin{split}
\tilde{T}_{ij} =& T_{ij}^{\mathcal{F}*},\\
\tilde{P}_{ij} =& T_{ij}^{\mathcal{F}*} + \min_{u_j(\cdot)\in\mathcal{U}_j} \{t: x_j(t,u_j(\cdot),\alpha_{ij},\dot{x}_j^0)=\beta_{ij}\},
\end{split}
\end{align}
where $\dot{x}_j^0 = \dot{x}_j(T_{ij}^{\mathcal{F}*},u_j(\cdot),x_j(0),\dot{x}_j(0))$ with $\alpha_{ij} = x_j(T_{ij}^{\mathcal{F}*},u_j(\cdot),x_j(0),\dot{x}_j(0))$.
If $x_j(0)\geq \alpha_{ij}$, let $\tilde{T}_{ij} = T_{ij}^{\mathcal{F}*}=0$ and $\tilde{P}_{ij} = \min\{t: x_j(t,u_j(\cdot),x_j(0),\dot{x}_j(0))=\beta_{ij}\}.$

 For $(i,j)\in\mathcal{N}\setminus \mathcal{F}$, there exists the first operation $(i',j)\in\mathcal{F}$, and \begin{align}
\begin{split}\label{proof:upper_tildeT_NF}
\tilde{T}_{ij} = &{T}_{i'j}^{\mathcal{F}*} +\min_{u_j(\cdot)\in\mathcal{U}_j} \{t: x_j(t,u_j(\cdot),\alpha_{i'j},\dot{x}_j^0)=\alpha_{ij}\},\\
\tilde{P}_{ij} = &{T}_{i'j}^{\mathcal{F}*} +\min_{u_j(\cdot)\in\mathcal{U}_j} \{t: x_j(t,u_j(\cdot),\alpha_{i'j},\dot{x}_j^0)=\beta_{ij}\},
\end{split}
\end{align}
where $\dot{x}_j^0=\dot{x}_j({T}_{i'j}^{\mathcal{F}*},u_j(\cdot),x_j(0),\dot{x}_j(0))$ with $\alpha_{i'j} = x_j({T}_{i'j}^{\mathcal{F}*},u_j(\cdot),x_j(0),\dot{x}_j(0))$.

Notice that $\tilde{P}_{ij}$ is $P_{ij}(\tilde{\mathbf{T}})$ by \eqref{definition:p_nolast} and \eqref{definition:p_last}. We will show that $\tilde{\mathbf{T}}$ and $\tilde{\mathbf{k}}=\mathbf{k}^{*}$ is a feasible solution to Problem~\ref{problem:MINLP}.

First, we will show that $\tilde{\mathbf{T}}$ satisfies constraint~\eqref{minlp:RD} in Problem~\ref{problem:MINLP}. For all $(i,j)\in{\mathcal{F}}$, $\bar{R}_{ij} = R_{ij}$ and $\bar{D}_{ij}\leq D_{ij}$. Since ${T}_{ij}^{\mathcal{F}*}$ satisfies constraint~\eqref{milp2:RD}, we have ${R}_{ij}\leq \tilde{T}_{ij}$. For $(i,j)\in \mathcal{N}\setminus {\mathcal{F}}$, we define $\tilde{T}_{ij}$ as the minimum time to reach conflict area $i$, thereby $\tilde{T}_{ij} = R_{ij}(\tilde{\mathbf{T}})$ by \eqref{definition:RD_nofirst}. This establishes that $R_{ij}\leq \tilde{T}_{ij}$ for all $(i,j)\in\mathcal{N}$.

For constraint~\eqref{minlp:disjunctive} in Problem~\ref{problem:MINLP}, let us focus on \eqref{proof:upper_tildeT_F}, \eqref{proof:upper_tildeT_NF}, and Definition~\ref{definition:up_TP} given ${\mathbf{T}}^{\mathcal{F}*}$. If $x_j(0)<\alpha_{j,min}$, we have $\bar{T}_{ij}\leq \tilde{T}_{ij}$ and $\tilde{P}_{ij}\leq \bar{P}_{ij}$ because $\bar{T}_{ij}$ and $\bar{P}_{ij}$ are computed with the maximum and minimum speed, respectively. If $x_j(0)\geq \alpha_{j,min}$, for $(i',j)\in\mathcal{F}$ we have $ T_{i'j}^{\mathcal{F}*} = \bar{R}_{i'j}$ since $s^*_U=0$ and $\bar{R}_{i'j} = \bar{D}_{i'j}$. This implies that $\bar{T}_{ij} = \tilde{T}_{ij}$ and $\bar{P}_{ij} = \tilde{T}_{ij}$. Thus, constraint~\eqref{milp2:disjunctive} becomes \begin{align*}
&\tilde{P}_{ij}\leq \bar{P}_{ij} \leq \bar{T}_{ij'} + M(1-{k}^*_{ijj'}) \leq \tilde{T}_{ij'} + M(1-{k}^*_{ijj'}),\\
& \tilde{P}_{ij'}\leq \bar{P}_{ij'} \leq \bar{T}_{ij} + M(1-{k}^*_{ij'j}) \leq \tilde{T}_{ij} + M(1-{k}^*_{ij'j}).
\end{align*}
That is, $\tilde{\mathbf{T}}$ and $\tilde{\mathbf{k}}=\mathbf{k}^*$ satisfy constraint \eqref{minlp:disjunctive} in Problem~\ref{problem:MINLP}.

Now we have a feasible solution $\tilde{\mathbf T}$ and $\tilde{\mathbf k}$. For $(i,j)\in\mathcal{N}\setminus \mathcal{F}$, we have $\tilde{T}_{ij} = R_{ij}\leq D_{ij}$, and thus, $\max_{(i,j)\in\mathcal{N}\setminus\mathcal{F}} (\tilde{T}_{ij} - D_{ij},0) = 0$. For $(i,j)\in\mathcal{F}$, we have $\bar{D}_{ij}\leq D_{ij}$ and the following inequalities complete the proof.
\begin{align*}
s^* &\leq \max_{(i,j)\in\mathcal{N}}(\tilde{T}_{ij}-{D}_{ij},0) =\max_{(i,j)\in\mathcal{F}}(\tilde{T}_{ij}-{D}_{ij},0)\\ &\leq \max_{(i,j)\in\mathcal{F}}(\tilde{T}_{ij}-\bar{D}_{ij},0) = s_U^*.
\end{align*}
Therefore, if $s_U^*=0$, we have $s^*=0$.
\end{proof}

 
By Theorems~\ref{theorem:lower} and \ref{theorem:upper}, we have $$s_L^* > 0 \Rightarrow s^* > 0 ~~\text{and}~~ s_U^*=0 \Rightarrow s^* = 0.$$
That is, from Problems~\ref{problem:milp_first} and \ref{problem:milp_upper}, which can be solved with a commercial solver such as CPLEX, we can find the solution for Problem~\ref{problem:MINLP} as shown in Table~\ref{table:summary}. However, when $s_L^*=0$ and $s_U^*>0$, represented by Case II in the table, $s^*$ is not exactly determined. In this case, we will provide approximation bounds.

\begin{table}[t!]
\centering
\caption{Summary of Theorems~\ref{theorem:lower} and \ref{theorem:upper}}
\label{table:summary}
\begin{tabular}[h!]{|l|c|c||c|}
\hline
& $s^*_L$ & $s^*_U$& $s^*$ \\
\hline
Case I & $\cdot$ & 0 & \textbf{0}\\
Case II & 0 & + & \textbf{?}\\
Case III & + & $\cdot$ & \textbf{+}\\
\hline
\end{tabular}
\end{table}

\begin{figure*}[t!]
	\centering
	\subfloat[]{	
		\includegraphics[width=0.23\linewidth]{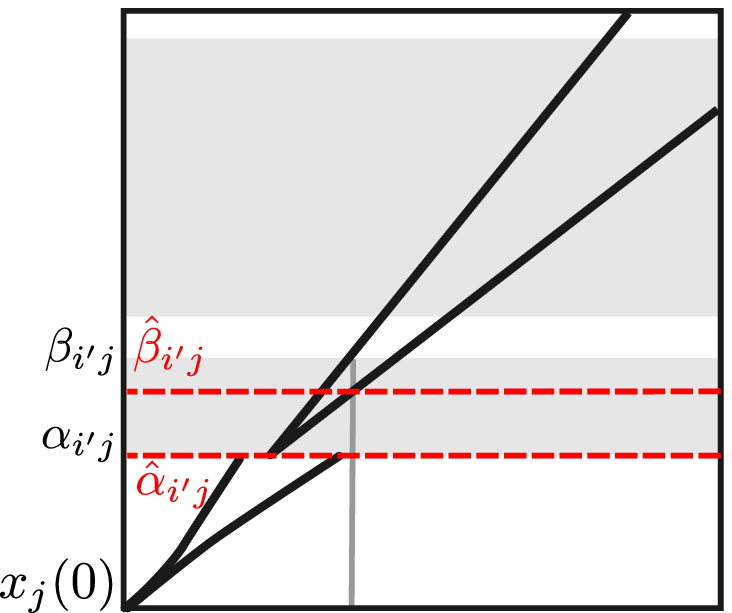}}
	\subfloat[]{	
		\includegraphics[width=0.23\linewidth]{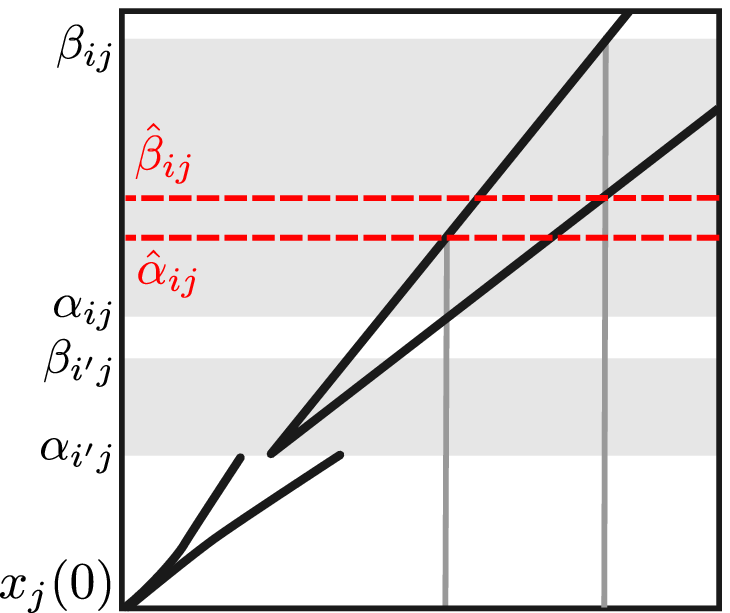}}
	\quad
	\subfloat[]{	
		\includegraphics[width=0.23\linewidth]{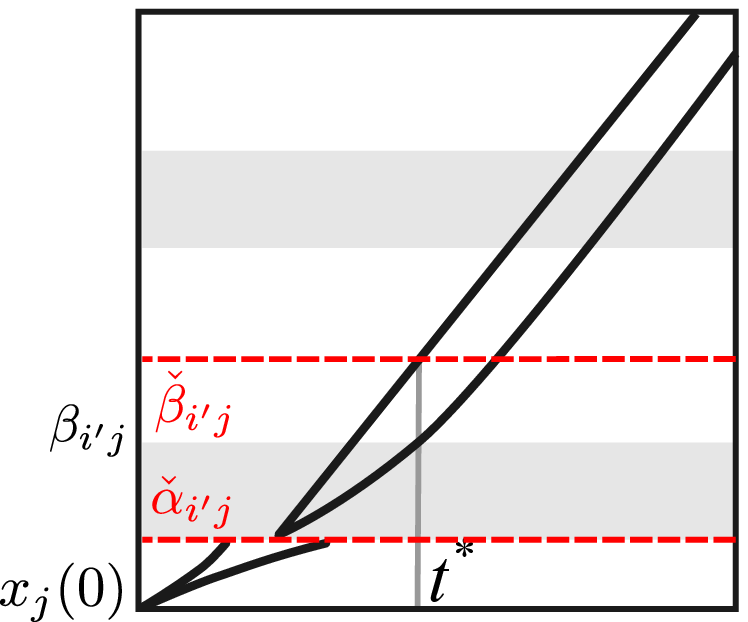}}
	\subfloat[]{	
		\includegraphics[width=0.23\linewidth]{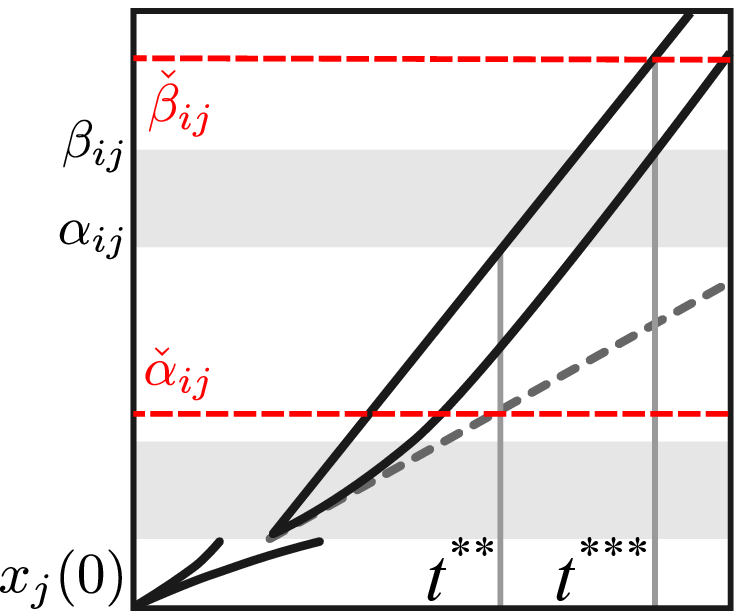}}
	\caption{Shrunk and Inflated conflict areas  for $(i',j)\in{\mathcal{F}}$ and $(i',j)\rightarrow (i,j)\in\mathcal{C}$. Figures (a)-(d) illustrate \eqref{definition:shrunk_F}-\eqref{definition:inflated_NF}, respectively. By definition, $ (\hat{\alpha}_{i'j}, \hat{\beta}_{i'j})\subseteq(\alpha_{i'j},\beta_{i'j})\subseteq (\check{\alpha}_{i'j}, \check{\beta}_{i'j})$ and $ (\hat{\alpha}_{ij}, \hat{\beta}_{ij})\subseteq(\alpha_{ij}, \beta_{ij})\subseteq(\check{\alpha}_{ij},\check{\beta}_{ij})$.}
	\label{figure:approach1_inflated}
\end{figure*}

\subsection{Approximation bounds}

We exactly solve Problem~\ref{problem:MINLP} in cases I and III as noted in Table~\ref{table:summary}, but have to approximate the solution in case II. In this section, we will focus on the latter case when $s_L^* = 0$ and $s_U^*>0$ and quantify the approximation bounds. Notice that $s_L^*>0$ and $s_U^*=0$ cannot occur.

We will prove the following statement: if $s_L^*=0$ and $s_U^*>0$, an input exists that makes the system avoid a \textit{shrunk bad set} but not an \textit{inflated bad set}. These bad sets are defined independent of time and thus we consider $\mathcal{N} = \bar{\mathcal{N}}$. 

A shrunk conflict area and an inflated conflict area are defined in the following definitions. See Figure~\ref{figure:approach1_inflated}.

\begin{definition}
A shrunk conflict area $(\hat{\alpha}_{ij}, \hat{\beta}_{ij})$ for all $(i,j)\in\mathcal{N}$ is defined as follows. For all $(i,j)\in\mathcal{F}$, \begin{align}
\begin{split}\label{definition:shrunk_F}
&\hat{\alpha}_{ij} = \alpha_{ij}, \\
&\hat{\beta}_{ij} = \min_{u_j(\cdot)\in\mathcal{U}_j} x_j\left(\frac{\beta_{ij}-\alpha_{ij}}{\dot{x}_{j,max}},u_{j}(\cdot), \alpha_{ij}, \dot{x}_{j,min}\right).
\end{split}
\end{align}

For all $(i,j)\in {\mathcal{N}}\setminus{\mathcal F}$, there exists the first operation $(i',j)$ for $i\ne i'$ such that $(i',j)\in {\mathcal{F}}$. 
\begin{align}
\begin{split}\label{definition:shrunk_NF}
\hat{\alpha}_{ij} &= \max_{u_j(\cdot)\in\mathcal{U}_j}  x_j\left(\frac{\alpha_{ij}-\alpha_{i'j}}{\dot{x}_{j,min}},u_j(\cdot),\alpha_{i'j},\dot{x}_{j,max}\right),\\
\hat{\beta}_{ij}&=\min_{u_j(\cdot)\in\mathcal{U}_j} x_j\left(\frac{\beta_{ij}-{\alpha}_{i'j}}{\dot{x}_{j,max}},u_j(\cdot),{\alpha}_{i'j},\dot{x}_{j,min}\right).
\end{split}
\end{align}
If $\hat{\alpha}_{ij} \geq \hat{\beta}_{ij}$, set $(\hat{\alpha}_{ij},\hat{\beta}_{ij})$ as an empty set.
\end{definition}

\begin{definition}\label{definition:inflated}
	An inflated conflict area $(\check{\alpha}_{ij}, \check{\beta}_{ij})$ for all $(i,j)\in {\mathcal{N}}$ is defined as follows.	For all $(i,j)\in{\mathcal F}$, 
	\begin{align}\label{definition:inflated_F}
	\check{\alpha}_{ij} = \alpha_{ij},&&\check{\beta}_{ij} = \max_{u_j(\cdot)\in\mathcal{U}_j} x_j({t}^{*},u_j(\cdot),\alpha_{ij},\dot{x}_{j,max}),
	\end{align}
	where ${t}^{*} = \min_{u_j(\cdot)\in\mathcal{U}_j} \{t:x_j(t,u_j(\cdot),\alpha_{ij},\dot{x}_{j,min})=\beta_{ij}\}$ and $\alpha_{ij} = \alpha_{j,min}$. 
	
	For all $(i,j)\in {\mathcal{N}}\setminus\mathcal{F}$,
	\begin{align}
	\begin{split}\label{definition:inflated_NF}
	\check{\alpha}_{ij} &= \min_{u_j(\cdot)\in\mathcal{U}_j}  x_j({t}^{**},u_j(\cdot),\alpha_{j,min},\dot{x}_{j,min}),\\
	\check{\beta}_{ij}&=\max_{u_j(\cdot)\in\mathcal{U}_j} x_j({t}^{***},u_j(\cdot),\alpha_{j,min},\dot{x}_{j,max}),
	\end{split}
	\end{align}
	where ${t}^{**}=\min_{u_j\in\mathcal{U}_j}\{t:x_j(t,u_j(\cdot),\alpha_{j,min},\dot{x}_{j,max})=\alpha_{ij}\}$ and ${t}^{***} = \min_{u_j\in\mathcal{U}_j} \{t:x_j(t,u_j(\cdot),\alpha_{j,min},\dot{x}_{j,min})=\beta_{ij}\}.$
	Notice that $t^*, t^{**},$ and $t^{***}$ are the same as the added times in \eqref{definition:milp2_T_F} and \eqref{definition:milp2_T_NF}.
\end{definition}

A shrunk bad set $\hat{\mathcal{B}}$ and an inflated bad set $\check{\mathcal B}$ are defined as follows:
\begin{align}
\begin{split}\label{equation:shrunkB}
\hat{\mathcal{B}}:=\{\mathbf{x}\in \mathbf{X}: \text{for some}~(i,j)\leftrightarrow (i,j')\in \mathcal{D}\\
x_j\in (\hat{\alpha}_{ij},\hat{\beta}_{ij})~\text{and}~x_{j'}\in (\hat{\alpha}_{ij'}, \hat{\beta}_{ij'})\}.
\end{split}
\end{align}
\begin{align}
\begin{split}\label{equation:inflatedB}
\check{\mathcal{B}}:=\{\mathbf{x}\in \mathbf{X}: \text{for some}~(i,j)\leftrightarrow (i,j')\in \mathcal{D}\\
x_j\in (\check{\alpha}_{ij},\check{\beta}_{ij})~\text{and}~x_{j'}\in (\check{\alpha}_{ij'}, \check{\beta}_{ij'})\}.
\end{split}
\end{align}
It can be checked that $$\hat{\mathcal{B}}\subseteq \mathcal{B}\subseteq\check{\mathcal{B}}$$ by showing that $(\hat{\alpha}_{ij},\hat{\beta}_{ij})\subseteq (\alpha_{ij}, \beta_{ij})\subseteq (\check{\alpha}_{ij}, \check{\beta}_{ij})$ for all $(i,j)\in\mathcal{N}$.

In the following theorems, we prove that the shrunk and inflated bad sets can represent the approximation errors of the solutions of Problems~\ref{problem:milp_first} and \ref{problem:milp_upper}, respectively. More precisely, we prove that 1) if the solution of Problem~\ref{problem:milp_first} is \textit{yes}, that is, $s_L^*=0$, then there exists an input that makes the system avoid the shrunk bad set, and 2) if the solution of Problem~\ref{problem:milp_upper} is \textit{no}, that is, $s_U^*>0$, then there is no input that makes the system avoid the inflated bad set.

\begin{theorem}\label{theorem:shrunkBadset}
Given an initial condition $(\mathbf{x}(0),\dot{\mathbf x}(0))$, if $s_L* =0$, then there exists an input signal $\mathbf{u}(\cdot)\in\mathcal{U}$ such that $\mathbf{x}(t,\mathbf{u}(\cdot),\mathbf{x}(0),\dot{\mathbf x}(0))\notin \hat{\mathcal{B}}$ for all $t\geq 0$.
\end{theorem}
\begin{proof}
If $s_L^*=0$, there exists an optimal solution $\mathbf{t}^*,\mathbf{p}^*$ and $\mathbf{k}^*$ that satisfies $t^*_{ij}\leq d_{ij}$ for all $(i,j)\in\mathcal{N}$ and constraints~\eqref{milp1:RD}-\eqref{milp1:disjunctive} in Problem~\ref{problem:milp_first}. 

For $(i,j)\in\mathcal{F}$, $r_{ij}\leq t^{*}_{ij}\leq d_{ij}$ implies that there exists $u_j^*(\cdot):[0, t^*_{ij})\rightarrow {U}_j$ such that \begin{equation}\label{proof:approximation_sL_Ft}
x_j(t^*_{ij},u_j^*(\cdot),x_j(0),\dot{x}_j(0))=\alpha_{ij}=\hat{\alpha}_{ij}
\end{equation} since the flow of $x_j$ is a continuous function of $u_j^*(\cdot)$ and the input space is path connected by Assumption~\ref{assumption:path-connected}. Let $\dot{x}_j(t^*_{ij})$ denote $\dot{x}_j(t^*_{ij},u_j^*(\cdot),x_j(0),\dot{x}_j(0))$. Constraint \eqref{milp1:process}, $(\beta_{ij}-\alpha_{ij})/\dot{x}_{j,max}\leq p^*_{ij}-t^*_{ij}\leq (\beta_{ij}-\alpha_{ij})/\dot{x}_{j,min}$, implies that for any input $u_j^*(\cdot):[t^*_{ij}, p^*_{ij})\rightarrow U_j$, \begin{equation}\label{proof:approximation_sL_Fp}
x_j(p^*_{ij}-t^*_{ij},u_j^*(\cdot),\alpha_{ij},\dot{x}_j(t^*_{ij}))\geq \hat{\beta}_{ij},
\end{equation} 
because $\hat{\beta}_{ij}$ defined in \eqref{definition:shrunk_F} is the minimum distance from $\alpha_{ij}$ traveled for time $(\beta_{ij}-\alpha_{ij})/ \dot{x}_{j,max}$. 

For $(i,j)\in\mathcal{N}\setminus \mathcal{F}$, there exists the first operation $(i',j)\in\mathcal{F}$. By the definition of $r_{ij}$ and $d_{ij}$ in \eqref{definition:low_RD_NF}, we have  $$t^*_{i'j}+\frac{\alpha_{ij}-\alpha_{i'j}}{\dot{x}_{j,max}}\leq r_{ij}\leq t^*_{ij}\leq d_{ij}\leq t^*_{i'j}+ \frac{\alpha_{ij}-\alpha_{i'j}}{\dot{x}_{j,min}}.$$ Since $\hat{\alpha}_{ij}$ is the maximum distance from $\alpha_{i'j}$ traveled for time $(\alpha_{ij}-\alpha_{i'j})/\dot{x}_{j,min}$, for any input $u_j^*(\cdot):[t_{i'j}^*, t_{ij}^*)\rightarrow {U}_j$, 
\begin{equation}\label{proof:approximation_sL_Nt}
x_j(t^*_{ij}-t^*_{i'j},u_j^*(\cdot),\alpha_{i'j},\dot{x}_j(t^*_{i'j}))\leq \hat{\alpha}_{ij}.
\end{equation} 
Similarly for any input $u_j^*(\cdot):[t_{i'j}^*, p_{ij}^*)\rightarrow {U}_j$, 
\begin{equation}\label{proof:approximation_sL_Np}
x_j(p^*_{ij}-t^*_{i'j},u_j^*(\cdot),\alpha_{i'j},\dot{x}_j(t^*_{i'j}))\geq \hat{\beta}_{ij}.
\end{equation} 
Thus, there exists an input $u_j^*(\cdot)\in\mathcal{U}_j$ that satisfies \eqref{proof:approximation_sL_Ft}-\eqref{proof:approximation_sL_Np}.

Now we will show that $\mathbf{x}(t,\mathbf{u}^*(\cdot),\mathbf{x}(0),\dot{\mathbf x}(0))\notin \hat{\mathcal{B}}$ for all $t\geq 0$. By \eqref{milp1:disjunctive}, we have for $(i,j)\leftrightarrow (i,j')\in\mathcal{D}$ either $p^*_{ij}\leq t_{ij'}^*$ or $p^*_{ij'}\leq t^*_{ij}$. Without loss of generality, we consider $p^*_{ij}\leq t_{ij'}^*$. Then, at time $t_{ij'}^*$, vehicle $j'$ has not yet entered shrunk conflict area $i$ as shown in \eqref{proof:approximation_sL_Ft} and \eqref{proof:approximation_sL_Nt}. At the same time, vehicle $j$ has already exited the shrunk conflict area as shown in \eqref{proof:approximation_sL_Fp} and \eqref{proof:approximation_sL_Np}. Thus, two vehicles never meet inside the same shrunk conflict area, and thus the system avoids the shrunk bad set with the input signal ${\mathbf u}^*(\cdot)$. 
\end{proof}

\begin{theorem}\label{theorem:inflatedBadset}
Given an initial condition $(\mathbf{x}(0),\dot{\mathbf x}(0))$, if $s_U^* >0$, then for all input signal $\mathbf{u}(\cdot)\in\mathcal{U}$, $\mathbf{x}(t,\mathbf{u}(\cdot),\mathbf{x}(0),\dot{\mathbf x}(0))\in \check{\mathcal{B}}$ for some $t\geq 0$.
\end{theorem}
\begin{proof}
We prove the contra-position: if there exists an input signal ${\mathbf{u}}^*(\cdot)\in\mathcal{U}$ such that $\mathbf{x}(t,{\mathbf{u}}^*(\cdot),\mathbf{x}(0),\dot{\mathbf{x}}(0))\notin \check{\mathcal{B}}$ for all $t\geq 0$, then $s_U^*=0$. 

For all $(i,j)\in\mathcal{N}$, let us define 
$${T}^*_{ij} :=\{t:x_j(t,u^*_j(\cdot),x_j(0),\dot{x}_j(0))=\check{\alpha}_{ij}\},$$
if $x_j(0) < \check{\alpha}_{ij}$. Otherwise, set $T^*_{ij} = 0$.
Also, if $x_j(0)<\check{\beta}_{ij}$,
$${P}^*_{ij} := \{t:x_j(t,{u}^*_j(\cdot),x_j(0),\dot{x}_j(0))=\check{\beta}_{ij}\}.$$
Otherwise, set $P^*_{ij}=0$.

Since $\mathbf{x}(t,{\mathbf{u}}^*(\cdot),\mathbf{x}(0),\dot{\mathbf{x}}(0))\notin \check{\mathcal{B}}$ for all $t\geq 0$, we have for all $(i,j)\leftrightarrow (i,j')\in\mathcal{D}$,
\begin{equation}\label{proof:no_milp_disjoint}
(T^*_{ij},P^*_{ij}) \cap (T^*_{ij'}, P^*_{ij'})=\emptyset,
\end{equation}
which indicates that each inflated conflict area is occupied by only one vehicle at a time.

Let us define a schedule $\tilde{\mathbf T}^{\mathcal F} = \{\tilde{T}_{ij}^{\mathcal F}: \text{for all}~(i,j)\in{\mathcal{F}}\}$ as follows: if $x_j(0) < \alpha_{j,min}$,
$$\tilde{T}_{ij}^{\mathcal F} := \{t:x_j(t,{u}^*_j(\cdot),x_j(0),\dot{x}_j(0))=\alpha_{j,min}\}.$$ If $\alpha_{j,min}\leq x_j(0)$, set $\tilde{T}^{\mathcal F}_{ij} =\bar{R}_{ij}$. By Definition~\ref{definition:up_RD}, $\tilde{T}_{ij}^{\mathcal F}\in [\bar{R}_{ij}, \bar{D}_{ij}]$. 
Thus, $\tilde{\mathbf{T}}^{\mathcal F}$ satisfies constraint~\eqref{milp2:RD} and $\tilde{T}_{ij}^{\mathcal F}-\bar{D}_{ij} \leq 0$ for all $(i,j)\in\mathcal{F}$.
If $\tilde{\mathbf{T}}^{\mathcal F}$ satisfies constraint~\eqref{milp2:disjunctive}, $\tilde{\mathbf{T}}^{\mathcal F}$ is a feasible solution to Problem~\ref{problem:milp_upper} with a corresponding binary variable $\tilde{\mathbf{k}}=\{\tilde{k}_{ijj'}:\text{for all}~(i,j)\leftrightarrow (i,j')\in\mathcal{D}\}$ and thus $s_U^*=0$. 

Given $\tilde{\mathbf{T}}^{\mathcal F}$, $\bar{\mathbf T}=\{\bar{T}_{ij}:\text{for all}~(i,j)\in\mathcal{N}\}$ and $\bar{\mathbf P}=\{\bar{P}_{ij}:\text{for all}~(i,j)\in\mathcal {N}\}$ are defined according to Definition~\ref{definition:up_TP}. 

First, consider $x_j(0)<\alpha_{j,min}$. For  $(i,j)\in{\mathcal F}$, we have by \eqref{definition:milp2_T_F}, $\bar{T}_{ij} = \tilde{T}_{ij}^{\mathcal F}$ and $\bar{P}_{ij}=\tilde{T}_{ij}^{\mathcal F}+{t}^*$ where ${t}^*$ is introduced in \eqref{definition:inflated_F}. Since $\check{\alpha}_{ij}\leq \alpha_{j,min}$ by definition, $T^*_{ij} \leq \bar{T}_{ij}$. Also, since $\check{\beta}_{ij}$ represents the maximum distance from $\alpha_{j,min}$ that vehicle $j$ can travel during $t^*$, traveling from $\alpha_{j,min}$ to $\check{\beta}_{ij}$ takes no less time than $t^*$. Thus, $\bar{P}_{ij} \leq P^*_{ij}$. For $(i,j)\in\mathcal{N}\setminus{\mathcal{F}}$, there exists $(i',j)\in {\mathcal{F}}$. By \eqref{definition:milp2_T_NF}, $\bar{T}_{ij} = \tilde{T}_{i'j}^{\mathcal F} + {t}^{**}$ and $\bar{P}_{ij} = \tilde{T}_{i'j}^{\mathcal F} + {t}^{***},$
where ${t}^{**}$ and ${t}^{***}$ are introduced in \eqref{definition:inflated_NF}. Since $\check{\alpha}_{ij}$ and $\check{\beta}_{ij}$ are the minimum and maximum distance traveled from $\alpha_{j,min}$ for time ${t}^{**}$ and ${t}^{***}$, respectively, we have $T^*_{ij} \leq \bar{T}_{ij}$ and $\bar{P}_{ij} \leq P^*_{ij}$.

Next, consider $\alpha_{j,min}\leq x_j(0)$. In this case, $\bar{R}_{ij} = \bar{D}_{ij}$ for all vehicle $j$'s operations by \eqref{definition:milp2_RD_inside} and thus for $(i',j)\in\mathcal{F}$, we have $\tilde{T}_{i'j}^{\mathcal{F}} = \bar{R}_{i'j}$. By \eqref{definition:milp2_T_F_inside}, $\bar{T}_{i'j}=\tilde{T}_{i'j}^{\mathcal{F}}$ and $\bar{P}_{i'j}$ is the minimum time to reach $\beta_{ij}$ since $\tilde{T}_{i'j}^\mathcal{F}-\bar{R}_{i'j} = 0$. The fact that $\check{\alpha}_{ij}\leq \alpha_{ij}$ and $\beta_{ij}\leq \check{\beta}_{ij}$ implies $T_{ij}^*\leq \bar{T}_{ij}$ and $\bar{P}_{ij}\leq P_{ij}^*$, respectively. For $(i,j)\in\mathcal{N}\setminus \mathcal{F}$, $\bar{T}_{ij}$ and $\bar{P}_{ij}$ are the minimum time to reach $\alpha_{ij}$ and $\beta_{ij}$, respectively, starting from $x_j(0)$. Since $\bar{T}_{ij}\geq (\alpha_{ij}-x_j(0))/\dot{x}_{j,max}$ and $\check{\alpha}_{ij}$ is the minimum distance traveled in $(\alpha_{ij}-x_j(0))/\dot{x}_{j,max}$, $T^*_{ij} \leq \bar{T}_{ij}$. Also, $\beta_{ij}\leq \check{\beta}_{ij}$ implies $\bar{P}_{ij}\leq P^*_{ij}$. 

Therefore, $(\bar{T}_{ij},\bar{P}_{ij}) \subseteq ( T^*_{ij}, P^*_{ij})$ for all $(i,j)\in\mathcal{N}$. By \eqref{proof:no_milp_disjoint}, we derive $(\bar{T}_{ij},\bar{P}_{ij})\cap (\bar{T}_{ij'},\bar{P}_{ij'})=\emptyset$ for all $(i,j)\leftrightarrow (i,j')\in\mathcal{D}$. This is equivalent to constraint~\eqref{milp2:disjunctive}, and thus there exists $\tilde{k}_{ijj'}$ and $\tilde{k}_{ij'j}$ satisfying the constraint.


Therefore, $\tilde{\mathbf{T}}^{\mathcal{F}}$ and $\tilde{\mathbf k}$ is a feasible solution with $s_U^*=0$.
\end{proof}

By Theorem~\ref{theorem:shrunkBadset}, if Problem~\ref{problem:milp_first} returns \textit{yes}, there is an input to make the system avoid the shrunk bad set. Otherwise, no input exists to make the system avoid the bad set. Similarly for Problem~\ref{problem:milp_upper}, if it returns \textit{yes}, there is an input to make the system avoid the bad set. Otherwise, no input exists to avoid the inflated bad set by Theorem~\ref{theorem:inflatedBadset}. Thus, the shrunk and inflated bad sets represent the over-approximation and under-approximation of the reachable set from an initial condition $(\mathbf{x}(0), \dot{\mathbf{x}}(0))$, respectively.

\subsection{Other upper bound solutions}
In Section~\ref{section:upper}, we formulate an MILP problem that yields an upper bound by considering the maximum input inside an intersection. Different MILP formulations are possible, for example, by considering the minimum input inside an intersection. To obtain a tighter upper bound, various combinations of the maximum and minimum inputs can be considered inside an intersection with a binary variable associated with each combination. At the expense of computational complexity, this approach is less conservative since more choices of inputs are allowed.

\section{Control design}\label{section:supervisor}
Based on the results of Section~\ref{section:approximateSolutions}, we can design a supervisor that is activated when a future collision is detected inside the inflated conflict areas.

Let \upperbound$(\mathbf{x}(0),\dot{\mathbf{x}}(0))$ be an algorithm solving Problem~\ref{problem:milp_upper} given an initial condition $(\mathbf{x}(0),\dot{\mathbf{x}}(0))$. Let \upperbound~return $(s_U^*,\mathbf{T}^{\mathcal{F}*})$ where $\mathbf{T}^{\mathcal{F}*}$ is the optimal solution.

The supervisory algorithm runs in discrete time with a time step $\tau$. At time $k\tau$, it receives the measurements of the state $(\mathbf{x}(k\tau), \dot{\mathbf{x}}(k\tau))$ and the desired input $\mathbf{u}_{d}^k\in \mathbf{U}$, which is a vector of inputs that the drivers are applying at the time. These measurements are used to predict the desired state at the next time step, which is denoted by $(\mathbf{x}^k_{d}, \dot{\mathbf{x}}^k_{d})$. We solve Problem~\ref{problem:milp_upper} to see if the desired state has a safe input signal within the approximation bound. 

If $\upperbound(\mathbf{x}^k_d,\dot{\mathbf{x}}^k_d)$ returns $\mathbf{T}^{\mathcal{F}*}$ that makes $s_U^*=0$, we can find a safe input signal by defining an input generator function $\sigma: \mathbf{X}\times\dot{\mathbf X}\times \mathbb{R}^n\rightarrow \mathcal{U}$ as follows:
\begin{align*}
&\sigma(\mathbf{x}^k_d,\dot{\mathbf x}^k_d,\mathbf{T}^{\mathcal{F}*}) \in \{\mathbf{u}(\cdot)\in\mathcal{U}:\,\text{for all}~ (i,j)\in\mathcal{F},\\
&x_j(T_{ij}^{\mathcal{F}*},
u_j(\cdot),x_{j,d}^k,\dot{x}^k_{j,d}) =\alpha_{ij}\,\text{and}\,u_j(t)=u_{j,max}\,\forall\,t\geq T_{ij}^{\mathcal{F}*}\},
\end{align*}
where $x_{j,d}^k$ and $\dot{x}_{j,d}^k$ denote the $j$-th entries of $\mathbf{x}^k_d$ and $\dot{\mathbf{x}}^k_d$, respectively. The supervisor stores this safe input restricted to time $(0,\tau)$ for a possible use at the next time step. Since there is an input signal that makes the system avoid entering the bad set from $(\mathbf{x}_d^k, \dot{\mathbf{x}}_d^k)$, the supervisor allows the desired input.

If $\upperbound(\mathbf{x}_d^k,\dot{\mathbf{x}}_d^k)$ returns $s_U^*>0$, the supervisor overrides the drivers with the safe input stored at the previous step. This safe input is used to predict the safe state, denoted by $(\mathbf{x}^k_{safe},\dot{\mathbf{x}}^k_{safe})$, and this safe state is used to generate a safe input for the next time step. We will prove in the next theorem that $\upperbound(\mathbf{x}^k_{safe},\dot{\mathbf{x}}^k_{safe})$ always returns $s_U^*=0$ and thus a safe input signal is defined. 

This supervisor is provided in Algorithm~\ref{algorithm:supervisor}. An input signal with superscript $k$, such as $\mathbf{u}^k(\cdot)$, denotes an input function from $[0,\tau)$ to $\mathbf{U}$. We define the desired input signal as $\mathbf{u}_d^k(\cdot):t \mapsto \mathbf{u}_d^k$ given  $\mathbf{u}_d^k\in\mathbf{U}$. Also, an input signal with superscript $k,\infty$ indicates that the domain of the input signal is $[0,\infty)$.

\begin{algorithm}[H]
	\caption{Supervisory control algorithm at $t=k\tau$}\label{algorithm:supervisor}
	\begin{algorithmic}[1]
		\Procedure{Supervisor}{$\mathbf{x}(k\tau),\dot{\mathbf x}(k\tau),\mathbf{u}_{d}^k$}
		\State $\mathbf{x}_{d}^k \leftarrow \mathbf{x}(\tau,\mathbf{u}_{d}^k(\cdot),\mathbf{x}(k\tau), \dot{\mathbf x}(k\tau))$
		\State $\dot{\mathbf{x}}_{d}^k \leftarrow \dot{\mathbf{x}}(\tau,\mathbf{u}_{d}^k(\cdot),\mathbf{x}(k\tau), \dot{\mathbf x}(k\tau))$
		\State $(s_U^*,\mathbf{T}^{\mathcal{F}*}_1)=\upperbound(\mathbf{x}_{d}^k, \dot{\mathbf{x}}_{d}^k)$
		\If{$s_U^*=0$}
		\State $\mathbf{u}^{k+1,\infty}_{safe}(\cdot)\leftarrow \sigma(\mathbf{x}_{d}^k, \dot{\mathbf{x}}_{d}^k,\mathbf{T}^{\mathcal{F}*}_1)$\label{algorithm:usafe1}
		\State $\mathbf{u}^{k+1}_{safe}(\cdot)\leftarrow \mathbf{u}^{k+1,\infty}_{safe}(t)~\text{for}~ t\in[0,\tau)$
		\State \textbf{return} $\mathbf{u}_{d}^k(\cdot)$\label{supervisor:return_desire}
		\Else
		\State $\mathbf{x}_{safe}^k \leftarrow \mathbf{x}(\tau,\mathbf{u}_{safe}^k(\cdot),\mathbf{x}(k\tau), \dot{\mathbf x}(k\tau))$
		\State $\dot{\mathbf{x}}_{safe}^k \leftarrow \dot{\mathbf{x}}(\tau,\mathbf{u}_{safe}^k(\cdot),\mathbf{x}(k\tau), \dot{\mathbf x}(k\tau))$
		\State $(\cdot,\mathbf{T}^{\mathcal{F}*}_2)=\upperbound(\mathbf{x}_{safe}^k, \dot{\mathbf{x}}_{safe}^k)$
		\State $\mathbf{u}^{k+1,\infty}_{safe}(\cdot)\leftarrow \sigma(\mathbf{x}_{safe}^k,\dot{\mathbf{x}}_{safe}^k,\mathbf{T}^{\mathcal{F}*}_2)$\label{algorithm:usafe2}
		\State $\mathbf{u}_{safe}^{k+1}(\cdot)\leftarrow \mathbf{u}_{safe}^{k+1,\infty}(t)~\text{for}~ t\in[0, \tau)$
		\State \textbf{return} $\mathbf{u}_{safe}^{k}(\cdot)$\label{supervisor:return_safe}
		\EndIf
		\EndProcedure
	\end{algorithmic}
\end{algorithm}

By Theorem~\ref{theorem:inflatedBadset}, $s_U^*>0$ indicates that no input signal exists to avoid the inflated bad set, that is, there may exist an input to avoid the bad set. Thus, we say this supervisor overrides vehicles when a future collision is detected within the approximation bound.

\begin{theorem}
Algorithm~\ref{algorithm:supervisor} is non-blocking, that is, if $\supervisor(\mathbf{x}(0),\dot{\mathbf{x}}(0),\mathbf{u}_d^0)\ne\emptyset$, then for any $\mathbf{u}_d^k\in \mathbf{U}$, $\supervisor(\mathbf{x}(k\tau),\dot{\mathbf{x}}(k\tau),\mathbf{u}_d^k)\ne\emptyset$.
\end{theorem}
\begin{proof}
	We prove this by induction on $k$. For the base case, assume $\supervisor(\mathbf{x}(0),\dot{\mathbf{x}}(0),\mathbf{u}_d^0)\ne\emptyset$ and $\mathbf{u}^{1,\infty}_{safe}(\cdot)$ is well-defined, that is, $\mathbf{u}^{1,\infty}_{safe}(\cdot)\in\mathcal{U}$. Suppose at time $(k-1)\tau$, $\supervisor(\mathbf{x}((k-1)\tau),\dot{\mathbf{x}}((k-1)\tau),\mathbf{u}_d^{k-1})$ is non-empty and $\mathbf{u}_{safe}^{k,\infty}$ is well-defined. At time $k\tau$, for any $\mathbf{u}_d^k\in\mathbf{U}$, $\supervisor(\mathbf{x}(k\tau),\dot{\mathbf{x}}(k\tau),\mathbf{u}_d^{k})$ is not empty since it returns either $\mathbf{u}_d^k(\cdot):t\mapsto \mathbf{u}_{d}^k$ or $\mathbf{u}_{safe}^k(\cdot):t\mapsto \mathbf{u}_{safe}^{k,\infty}(t)$. We need to show that $\mathbf{u}_{safe}^{k+1,\infty}(\cdot)$ is well-defined in lines \ref{algorithm:usafe1} or \ref{algorithm:usafe2}.
	
	In line~\ref{algorithm:usafe1}, since $\mathbf{T}_1^{\mathcal{F}*}=\{T_{ij,1}^{\mathcal{F}*}:(i,j)\in\mathcal{F}\}$ yields $s_U^*=0$, we have $\bar{R}_{ij}\leq T_{ij,1}^{\mathcal{F}*}\leq \bar{D}_{ij}$ for all $(i,j)\in\mathcal{F}$ and thus there exists $u_j(\cdot)\in \mathcal{U}_j$ that satisfies $x_j(T_{ij,1}^{\mathcal F*},u_j(\cdot),x_{j,d}^k, \dot{x}_{j,d}^k)=\alpha_{ij}$. Thus, $\sigma(\mathbf{x}_d^k, \dot{\mathbf{x}}_d^k,\mathbf{T}_1^{\mathcal{F}*})$ is well-defined.
	
	In line~\ref{algorithm:usafe2}, we will show that $\mathbf{T}_2^{\mathcal{F}*}$ satisfies $s_U^*=0$ and thus a safe input signal is well-defined. We have that $(\mathbf{x}(k\tau),\dot{\mathbf{x}}(k\tau))$ is either $(\mathbf{x}^{k-1}_d, \dot{\mathbf{x}}^{k-1}_d)$ or $(\mathbf{x}^{k-1}_{safe},\dot{\mathbf{x}}^{k-1}_{safe})$ depending on the output of the supervisor at the previous time step. Thus, at time $(k-1)\tau$, \upperbound$(\mathbf{x}(k\tau), \dot{\mathbf{x}}(k\tau))$ yielded $s_U^*=0$ with an optimal solution $\mathbf{T}^{\mathcal F, k-1}=\{T_{ij}^{\mathcal{F},k-1}:(i,j)\in\mathcal{F}\}$. Let $\bar{R}_{ij}^{k-1}, \bar{D}_{ij}^{k-1}, \bar{T}_{ij}^{k-1}$, and $\bar{P}_{ij}^{k-1}$ be the parameters used in the problem. Now, consider \upperbound$(\mathbf{x}_{safe}^k, \dot{\mathbf{x}}_{safe}^k)$ with parameters $\bar{R}_{ij}^k, \bar{D}_{ij}^k, \bar{T}_{ij}^k,$ and $\bar{P}_{ij}^k$. If we define $\tilde{T}_{ij}^{\mathcal{F}} = T_{ij}^{\mathcal{F},k-1}-\tau$ for all $(i,j)\in\mathcal{F}$, we have $\tilde{T}_{ij}^{\mathcal{F}}\in [\bar{R}_{ij}^k, \bar{D}_{ij}^k]$ since $T_{ij}^{\mathcal{F},k-1}\in [\bar{R}_{ij}^{k-1}, \bar{D}_{ij}^{k-1}]$ and $\mathbf{u}_{safe}^{k}(\cdot)$ is in $\mathcal{U}$. Also, since $(\bar{T}^k_{ij}, \bar{P}^k_{ij})\subseteq(\bar{T}^{k-1}_{ij}, \bar{P}^{k-1}_{ij})-\tau$ and $(\bar{T}^{k-1}_{ij},\bar{P}^{k-1}_{ij})\cap(\bar{T}^{k-1}_{ij'}, \bar{P}^{k-1}_{ij'})=\emptyset$ for all $(i,j)\leftrightarrow (i,j')\in\mathcal{D}$, we have $(\bar{T}^k_{ij},\bar{P}^k_{ij})\cap(\bar{T}^k_{ij'}, \bar{P}^k_{ij'})=\emptyset$. Therefore, $\tilde{\mathbf{T}}^{\mathcal{F}}=\{\tilde{T}_{ij}^{\mathcal{F}}: \text{for all}~(i,j)\in\mathcal{F}\}$ is a feasible solution that yields $s_U^*=0$, and thus there exists the optimal solution $\mathbf{T}_2^{\mathcal{F}*}$ satisfying $s_U^*=0$. Therefore, $\sigma(\mathbf{x}_{safe},\dot{\mathbf{x}}_{safe},\mathbf{T}_2^{\mathcal{F}*})$ is well-defined. 
\end{proof}

\section{Simulation Results}\label{section:simulation}

We implemented Algorithm~\ref{algorithm:supervisor} on the cases illustrated in Figures~\ref{figure:general_intersection} and \ref{figure:intersection} to validate its collision avoidance performance and its non-blocking property. We implemented the algorithm on MATLAB and performed simulations on a personal computer consisting of an Intel Core i7 processor at 3.10\,GHz and 8\,GB RAM.

In the simulations, we consider the vehicle dynamics with a quadratic drag term \cite{hafner_cooperative_2013} as follows: for all $j\in\{1,\ldots,n\}$
$$\ddot{x}_{j} = a u_j + b \dot{x}_j^2.$$
Also, the following parameters are used: $\tau = 0.1, a = 1, b = 0.005.$ For all $j\in\{1,\ldots,n\}$, $u_{j,max} = 2, u_{j,min}= -2, \alpha_{j,min}=20$. For all $(i,j)\in\mathcal N$, $\beta_{ij}-\alpha_{ij} = 5$ and for all $(i,j)\rightarrow (i',j)\in\mathcal C$, $\alpha_{i'j}-\alpha_{ij} = 6$.

\begin{figure}[h!]
	\centering
	\includegraphics[width = .9\columnwidth]{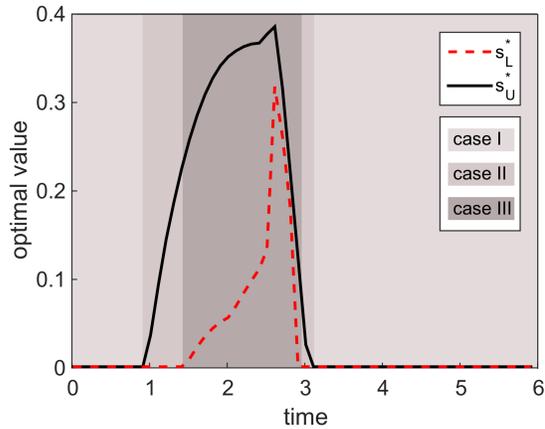}
	\caption{Simulation results without the supervisor (Algorithm~\ref{algorithm:supervisor}) for the scenario in Figure~\ref{figure:intersection}. Cases I, II, and III denote the same cases in Table~\ref{table:summary}.}
	\label{figure:simulation_bounds}
\end{figure}
\begin{figure*}[t!]
	\centering 
	\subfloat[Inflated bad set $\check{\mathcal{B}}$]{	
		\includegraphics[width=.33\linewidth]{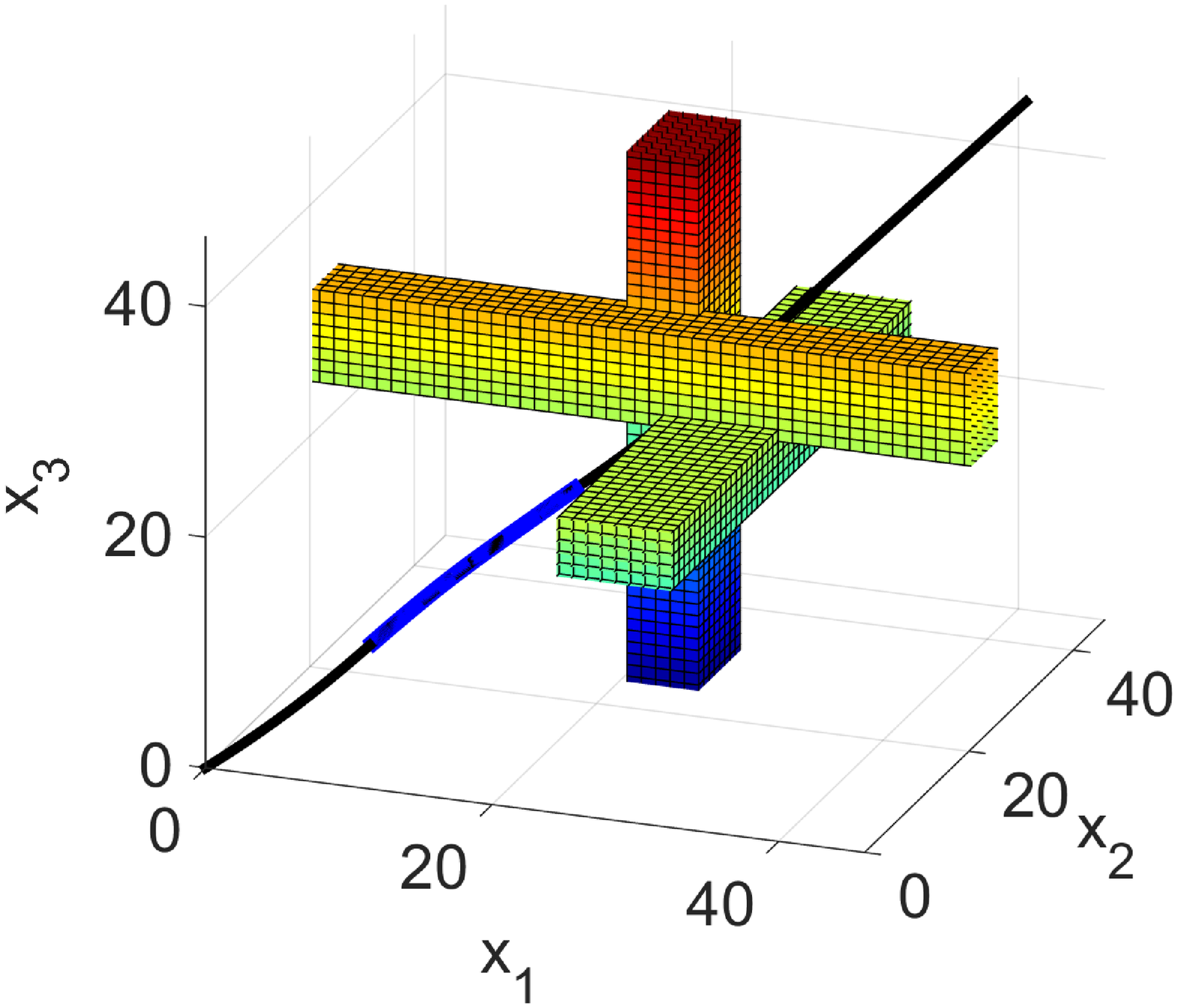}}
	\subfloat[Bad set $\mathcal B$]{	
		\includegraphics[width=.33\linewidth]{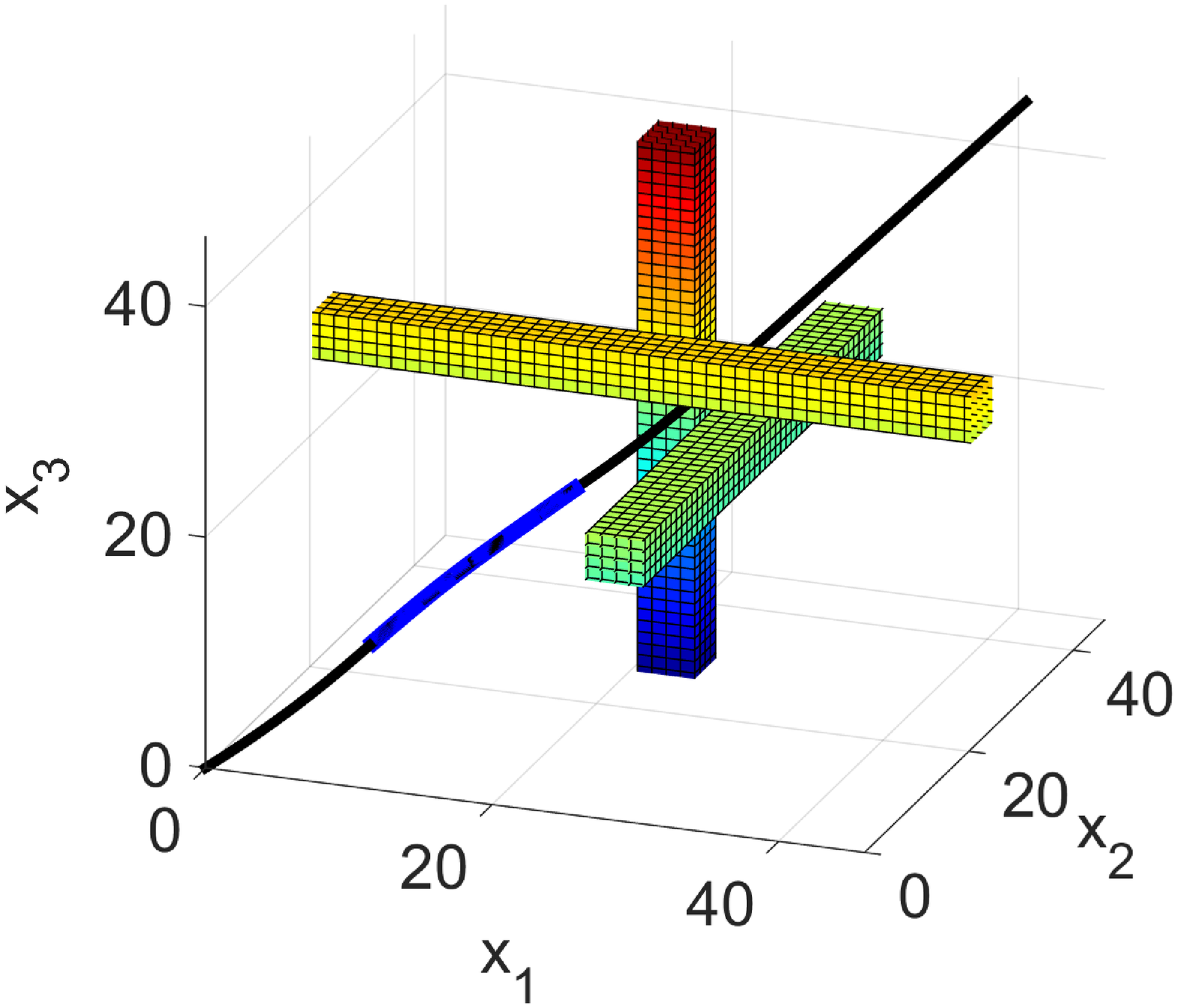}}
	\subfloat[Shrunk bad set $\hat{\mathcal B}$]{	
		\includegraphics[width=.33\linewidth]{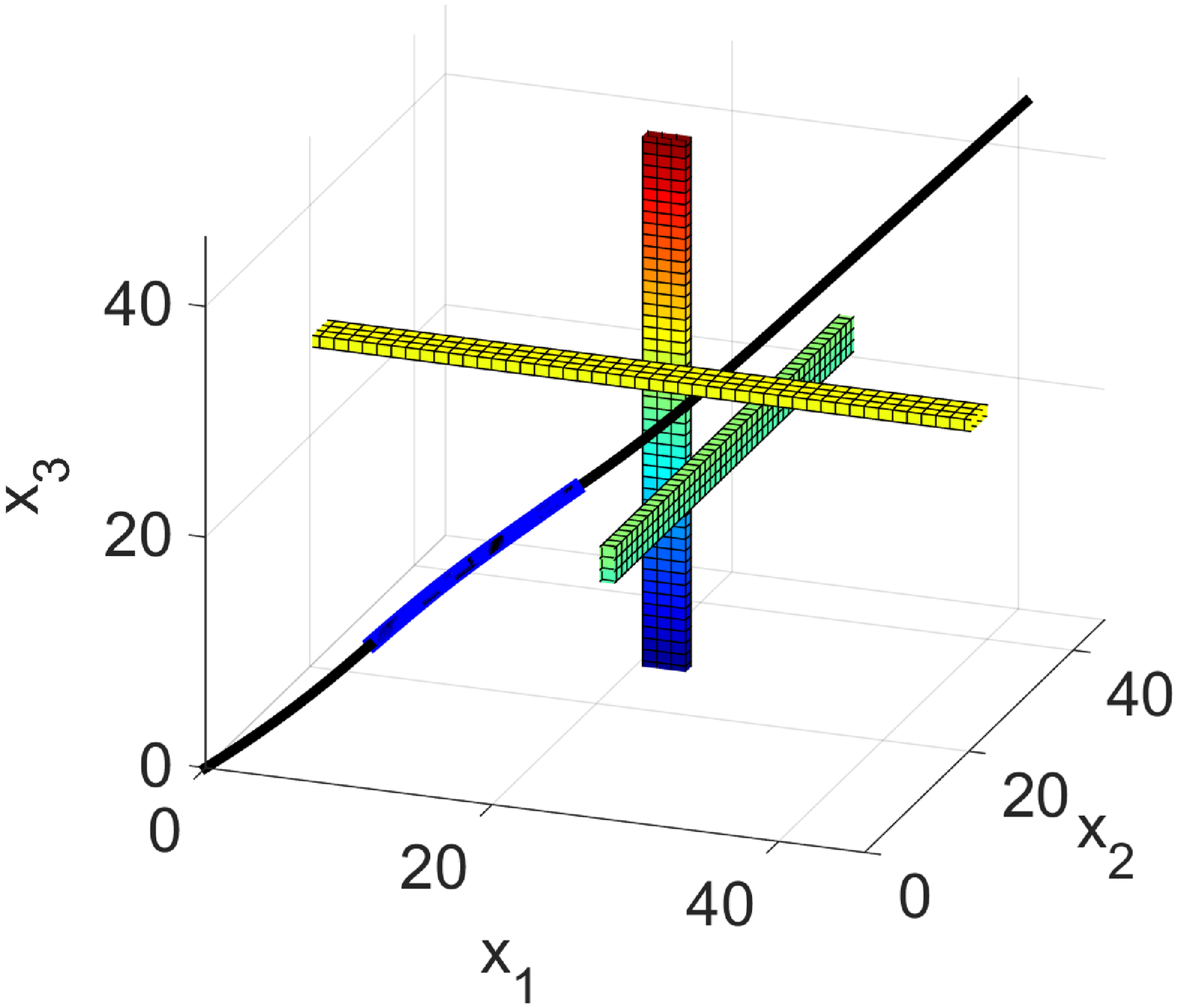}}
	\caption{Simulation results with the supervisor for the scenario in Figure~\ref{figure:intersection}. The black line represents the system trajectory and is the same on each figure. The line turns to blue when the supervisor intervenes to prevent a predicted collision. The solid in each figure is (a) the inflated bad set, (b) the bad set, and (c) the shrunk bad set. The supervisor manages the system to avoid entering the bad set.}
	\label{figure:simulation_badsets}
\end{figure*}

Let us consider the scenario illustrated in Figure~\ref{figure:intersection} with the following initial condition and parameters: $\mathbf{x}(0) = (0, 0, 0),\dot{\mathbf x}(0) = (10, 8, 8),$ and $\dot{x}_{j,min} = 8,\dot{x}_{j,max} = 10$ for all $j\in\{1\ldots,n\}$. Without implementing the supervisor (Algorithm~\ref{algorithm:supervisor}), we let the vehicles travel with the desired input $\mathbf{u}_{d}^k = (-2, -2, 2)$ for all $k$ and plot the optimal values of Problems~\ref{problem:milp_first} and \ref{problem:milp_upper}. These are shown in Figure~\ref{figure:simulation_bounds}. As proved in Theorems~\ref{theorem:lower} and \ref{theorem:upper}, $s_U^*=0$ implies $s_L^*=0$. The trajectory of the system with implementing the supervisor is shown in Figure~\ref{figure:simulation_badsets}(a)-(c). The trajectory (black line) is controlled by the supervisor when $s_U^*>0$ (the line is thicker in blue) so that it avoids the bad set (solid in (b)). Notice that the trajectory penetrates the inflated bad set (solid in (a)) but not the shrunk bad set (solid in (c)) as proved in Theorems~\ref{theorem:shrunkBadset} and \ref{theorem:inflatedBadset}.

Now, let us consider the scenario illustrated in Figure~\ref{figure:general_intersection} with the following initial condition and parameters: $\mathbf{x}(0) =(0,-2,5,-5,0,5,0,1,5,4,0,-2,5,5,0,5,-2,0,-2,0)$ and for all $j\in\{1,\ldots,n\}, \dot{x}_{j}(0)=5, \dot{x}_{j,min} = 1, \dot{x}_{j,max}=10$. With the desired input $\mathbf{u}_{d}^k = \mathbf{u}_{max}$ for all $k$, the result is shown in Figure~\ref{figure:simulation_results2}. The trajectory of vehicle 1 (black line) and the trajectories of other vehicles that share the same conflict area (red dotted lines) never stay inside the conflict area simultaneously. This implies that the supervisor overrides vehicles when necessary (when blue boxes appear) to make them cross the intersection without collisions. Because unnecessary to prove optimality, solving feasibility problems requires less computational effort than solving optimization problems \cite{cplex_2015}. That is, solving the following problem takes less computation time than solving Problem~\ref{problem:milp_upper}: given an initial condition, determine if there exists a feasible solution $(\mathbf{T}^{\mathcal F}, \mathbf{k})$ that satisfies \eqref{milp2:RD}, \eqref{milp2:disjunctive}, and $T_{ij}^{\mathcal{F}}\leq \bar{D}_{ij}$ for all $(i,j)\in\mathcal{F}$. Notice that this problem is equivalent to Problem~\ref{problem:milp_upper}. Based on the solution of this feasibility problem, Algorithm~\ref{algorithm:supervisor} takes no more than 0.05\,s per iteration, even in this realistic size scenario involving 20 vehicles, 48 conflict areas, and 120 operations on a representative geometry of dangerous intersections. Given that the allocated time step for intelligent transportation systems is 0.1\,s \cite{US:2015:ITS}, this algorithm can be implemented in real time.

\begin{figure}[t!]
	\centering
	\includegraphics[width = .9\columnwidth]{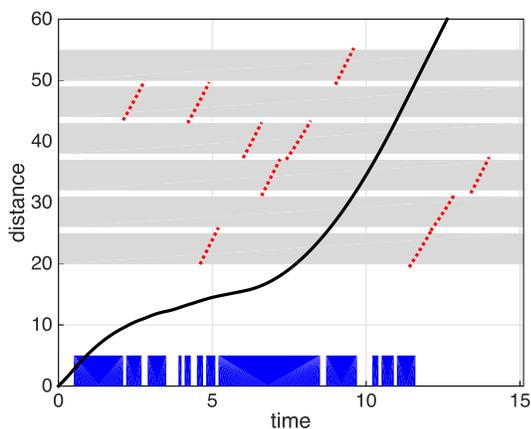}
	\caption{Trajectory of vehicle 1 in the scenario of Figure~\ref{figure:general_intersection} crossing six conflict areas. The blue boxes represent the times at which the supervisor overrides the vehicles. The red dotted lines are the trajectories of the other vehicles that share the same conflict area. This graph shows that each conflict area is used by only one vehicle at a time.}
	\label{figure:simulation_results2}
\end{figure}

\section{Conclusions}\label{section:concluisions}
In this paper, we presented the design of a supervisory algorithm that determines the existence of a future collision among vehicles at an intersection (safety verification) and overrides the drivers with a safe input if a future collision is detected (control design). We translated the safety verification problem into a scheduling problem by exploiting monotonicity of the system. This scheduling problem minimizes the maximum lateness and determines if the optimal cost is zero where the zero optimal cost corresponds to the case in which all vehicles can cross the intersection without collisions. Because of the nonlinear second-order dynamics of vehicles, the scheduling problem is a Mixed Integer Nonlinear Programming (MINLP) problem, which is computationally difficult to solve. We thus approximately solved this scheduling problem by solving two Mixed Integer Linear Programming (MILP) problems that yield lower and upper bounds of the optimal cost of the scheduling problem. We quantified the approximation error between the exact and approximate solutions to the scheduling problem. We presented the design of the supervisor based on the MILP problem that computes the upper bound and proved that it is non-blocking. Computer simulations validated that the supervisor can be implemented in real time applications.

While we assumed in this paper that there is only one vehicle per lane, our approach can be easily modified to deal with the case in which multiple vehicles are present on each lane. One possible modification can be solving the scheduling problem only for the first vehicles on lanes while letting the following vehicles maintain a safe distance from their front vehicles. Instead of this naive approach, we are currently investigating a less conservative approach. Also, Problem~\ref{problem:milp_upper} can be extended to include the presence of uncertainty sources, such as measurement noises, process errors, and not communicating vehicles, as done in our previous works \cite{bruni_robust_2013,ahn_supervisory_2014} for the single conflict area intersection model. Also, in future work, the assumption that the routes of vehicles are known in advance will be relaxed.

\bibliographystyle{IEEEtran}
\bibliography{../IEEEabrv}

\end{document}